\documentclass[11pt]{amsart}
\usepackage{amssymb,mathrsfs,graphicx,enumerate}
\usepackage{colortbl}
\definecolor{black}{rgb}{0.0, 0.0, 0.0}
\definecolor{red}{rgb}{1.0, 0.5, 0.5}

\title[   ]{Criteria on contractions for entropic discontinuities of  systems of conservation laws}

\author[Kang]{Moon-Jin Kang}
\address[Moon-Jin Kang]{\newline Laboratoire Jacques-Louis Lions, 
\newline UPMC F-75005 Paris, France
\newline Department of Mathematics, \newline The University of Texas at Austin, Austin, TX 78712, USA}
\email{moonjinkang@math.utexas.edu}

\author[Vasseur]{Alexis F. Vasseur}
\address[Alexis F. Vasseur]{\newline Department of Mathematics, \newline The University of Texas at Austin, Austin, TX 78712, USA}
\email{vasseur@math.utexas.edu}

\newtheorem{theorem}{Theorem}[section]
\newtheorem{lemma}{Lemma}[section]

\newtheorem{remark}{Remark}[section]

\newtheorem{definition}{Definition}[section]

\newcommand{\bbr}{\mathbb R}

\numberwithin{figure}{section}
\newcommand{\beq}{\begin{equation}}
\newcommand{\eeq}{\end{equation}}
\newcommand{\bsp}{\begin{split}}
\newcommand{\esp}{\end{split}}

\def\eps{\varepsilon }

\newcommand\adots{\mathinner{\mkern2mu\raise1pt\hbox{.}
\mkern3mu\raise4pt\hbox{.}\mkern1mu\raise7pt\hbox{.}}}

\def\charf {\mbox{{\text 1}\kern-.30em {\text l}}}

\begin{document}
\bibliographystyle{plain}

\date{\today}

\subjclass{    } 

\thanks{\textbf{Acknowledgment.} The work of M.-J. Kang was partially supported by Basic Science Research Program through the National Research Foundation of Korea funded by the Ministry of Education, Science and Technology (NRF-2013R1A6A3A03020506), and by the Foundation Sciences Math$\acute{\mbox{e}}$matiques de Paris as a postdoctoral fellowship.  The work of A. F. Vasseur was partially supported by the NSF Grant DMS 1209420.
}

\begin{abstract}
We study the contraction properties (up to shift) for admissible Rankine-Hugoniot discontinuities of  $n\times n$ systems of conservation laws endowed with a convex entropy. We first generalize the criterion developed in \cite{Serre-Vasseur}, using the spatially inhomogeneous pseudo-distance introduced in \cite{Vasseur-2013}. Our generalized criterion  guarantees the contraction property  for  extremal shocks of a large class of   systems, including the Euler system. Moreover, we introduce necessary conditions for contraction, specifically targeted for   intermediate shocks.  As an application, we show that intermediate shocks of the two-dimensional isentropic magnetohydrodynamics do not verify any of our contraction properties. We also investigate the contraction properties, for contact discontinuities of the Euler system, for a certain range of contraction weights.
 All results do not involve any smallness condition on the initial perturbation, nor on the size of the shock.
\end{abstract}
\maketitle \centerline{\date}


\tableofcontents

\section{Introduction}
\setcounter{equation}{0}
In this paper, we develop criteria for  the existence of contraction properties  of  admissible Rankine-Hugoniot discontinuities (typically entropic shocks and contact discontinuities) of a wide class of systems of $n$ conservation laws endowed with a convex entropy. Consider a $n\times n$ system of conservation laws
\begin{align}
\begin{aligned} \label{main}
&\partial_t u + \partial_x f(u) = 0, \quad t>0,~x\in \bbr,\\
&u(0,x) = u_0(x),
\end{aligned}
\end{align}
endowed with a strictly convex entropy $\eta$.\\ 

For  the scalar case ($n=1$), Kru$\check{\mbox{z}}$kov's theory \cite{K1} shows that the semi-group associated with \eqref{main} is contractive for the $L^1$ norm. For the system case, under small BV perturbation, Bressan, Liu and Yang in \cite{Bressan2, Liu-Yang} constructed a $L^1$ semi-group of solutions.  However, the $L^1$-contraction property does not hold, generally, for  systems (See Temple \cite{Temple}). \\ 

The Kru$\check{\mbox{z}}$kov's semi-group is not contractive in $L^p$ for $p>1$, unless the flux is linear. However, Leger showed in \cite{Leger}, that any perturbation of an entropic shock wave is contractive in $L^2$, up to shift. More precisely, for a strictly convex flux function $f$, and an associated entropic shock $(u_l,u_r,\sigma)$ (i.e., $u_l>u_r$), and for any bounded entropy solution $u$ of this scalar conservation law, there exists a Lipschitz shift $t\mapsto h(t)$ such that 
\[
\int_{\bbr} |u(t,x+h(t))-S(t,x)|^2 dx 
\]
is not increasing in time, where $S(t,x)$ is the traveling wave associated to   $(u_l,u_r,\sigma)$ 
\begin{align}
\begin{aligned}\label{shock-form}
S(t,x)=\left\{ \begin{array}{ll}
         u_l & \mbox{if $  x < \sigma t$},\\
         u_r & \mbox{if $ x> \sigma t$}.\end{array} \right.
\end{aligned}
\end{align}
The shift $h(t)$ depends on the solution $u$. This contraction property can be extended from $L^2$, to any relative entropy $\eta(u|S)$ associated to a convex entropy $\eta$ (see next section). Some extension to $L^p$, for $1<p<\infty$, can be found in Adimurthi, Goshal and Veerappa Gowda \cite{AGV1}.
\vskip0.1cm

In the case of systems ($n>1$), this kind of contraction property has been studied  in \cite{Serre-Vasseur, Vasseur-2013}. In \cite{Serre-Vasseur} the authors developed, in the case of systems,  a criterion for contraction of admissible (Rankine-Hugoniot) discontinuities. Their criterion is  satisfied, for instance,  by the Keyfitz-Kranzer system with a rotationally symmetric flux. However, it  is not applicable to many cases, including the Euler system. In \cite{Vasseur-2013}, the notion of contraction was extended to a family of non-homogenous pseudo-norms, defined for a fixed $a>0$, as
\begin{align}
\begin{aligned}\label{V-distance}
d(u(t,x),S(t,x))=\left\{ \begin{array}{ll}
         \eta (u(t,x)|u_l) & \mbox{if $  x < \sigma t$},\\
         a\eta (u(t,x)|u_r) & \mbox{if $ x>\sigma t$},\end{array} \right.
\end{aligned}
\end{align}
where $S$ is the traveling wave associated to the studied entropic shock. Notice that the case of $a=1$ corresponds to the case studied in  \cite{Serre-Vasseur}.  In \cite{Vasseur-2013}, it is  shown, that any extremal shock (i.e. $1$-shock or $n$-shock) verifies a contraction property, up to a shift, for such pseudo-norms with suitable weights $a>0$.
This pseudo-distance \eqref{V-distance} (determined by the weight $a$) does not depend on on the solution $u$. It depends only on the system and the traveling wave $S$. 
\vskip0.1cm

The purpose of this article is to generalize the criterion developed in \cite{Serre-Vasseur} to the spatially inhomogeneous pseudo-distance introduced  in \cite{Vasseur-2013}. We first apply our generalized criterion to the case of   extremal shocks. Then, we develop criteria specific for intermediate shocks, and intermediate contact discontinuities. 
We present two applications of those criteria. 
\vskip0.1cm

First, we show that intermediate shocks of the two-dimensional isentropic MHD (which is a   $4\times 4$ system), do not verify the contraction property, for any weight $a>0$.  For inviscid and viscous stability issues for the  MHD, we refer to \cite{B-H-Zumbrun, B-L-Zumbrun, G-M-W-Zumbrun, M-Zumbrun}. 
\vskip0.1cm 
For the contact discontinuities of the Euler system, it is shown in \cite{Serre-Vasseur_2015}, that the contraction property holds for the specific value $a=\theta_r/\theta_l$, the ratio of temperatures on the right, and on the left of the contact discontinuity. We show that this cannot hold for a large range  of other weights $a$. 
\vskip0.1cm

Our criteria depend only on  the structure of the system \eqref{main}, and the fixed admissible discontinuity. Contrary to the analysis in  \cite{Vasseur-2013}, it does not involve the study of every solutions $u$, nor the construction of the shift. This simplifies a lot its applicability.  The main difficulty of the analysis of \cite{Vasseur-2013}, for the contraction, is due to the construction of  the suitable shift and weights $a$. 
\vskip0.1cm    

The theory of contraction, based on the relative entropy, is valid for large perturbation, without smallness conditions. We consider any  bounded entropy weak solutions $u$ to \eqref{main} verifying a $BV_{loc}$ property. The  $BV_{loc}$ property is stronger than the strong trace property used in \cite{LV, Vasseur-2013}. All our results still hold under the assumption of the strong trace property instead of $BV_{loc}$. However, following \cite{Serre-Vasseur}, we restrict ourselves to the  $BV_{loc}$ case to simplify the exposition.  Note that the existence theory of  entropy weak solutions to the system \eqref{main} (when $n\ge 3$) for large data is  open. Strong trace property in the case of scalar conservation laws have been widely studied \cite{Chen_trace, DeLellis, Kwon, KV, Panov, Panov2, Vasseur_trace}. However, in the case of systems, the validity of the strong trace property is mainly an open problem. This has been shown only for the particular case of isentropic gas dynamics with $\gamma=3$, for traces in time, in \cite{Vasseur_gamma3}. 
\vskip0.1cm

Our analysis is based on the relative entropy method. It has been first used by Dafermos \cite{Dafermos1} and DiPerna \cite{DiPerna} to show the weak-strong uniqueness and stability of Lipschitz regular solution to conservation laws. (See also \cite{Dafermos4, Dafermos2})
We refer to \cite{CV, Kang-V-1, Leger, LV, Serre-2015, Serre-Vasseur_2015} for  applications of the relative entropy method to the stability of large perturbation in various contexts. This method is also an important tool in the study of asymptotic limits to conservation laws. Applications of the relative entropy method in this context began with the work of Yau \cite{Yau} and have been studied in various context (See for instance \cite{Bardos_Levermore_Golse1, Bardos_Levermore_Golse2, BTV, BV, SaintRaymond1, Kang-V-2, Lions_Masmoudi, Saintraymond3, MV, Saintraymond4}). 
\vskip0.1cm

The rest of the paper is organized as follows. In Section \ref{sec-criteria}, we present the criteria for the contraction of admissible discontinuities, and then identify the necessary conditions for the contraction. As an application of our criteria, in Section \ref{cont-shock}, we prove the contraction property of extremal shocks. In Section \ref{no-cont-shock}, using the necessary condition developed in section \ref{sec-criteria}, we construct two kinds of criteria preventing the contraction property for intermediate entropic shocks, and then present an application to MHD. It turns out that one of the criteria still hold true for intermediate contact discontinuities. In Section \ref{no-cont-inter}, as an application of this criterion to gas dynamics, we find a range of weights, for which there is no contraction for  2-contact discontinuities of the full Euler system.

\section{Preliminaries}
In this section, we present our framework, and basic concepts and properties, needed for the analysis in the following sections. 
\subsection{General framework}
We consider a $n\times n$ system of conservation laws:
\begin{align}
\begin{aligned}\label{main-1}
\partial_t u + \partial_x f(u) = 0, \quad t>0,~x\in \bbr,
\end{aligned}
\end{align}
which is endowed with a strictly convex entropy $\eta$, thus the system is hyperbolic on the state space where $\eta$ exists. Here, the flux $f$, the entropy $\eta$ and associated entropy flux $q$ are assumed to be all defined on an open convex state space $\mathcal{V}\subset \bbr^n$ and of class $\mathcal{C}^2(\mathcal{V})$, and the following compatibility relation holds on $\mathcal{V}$:
\[
\partial_j q =\sum_{i=1}^n \partial_i \eta \partial_j f_i,\quad 1\le j\le n.
\]
which is conventionally rewritten as 
\[\nabla q=\nabla\eta \nabla f, \]
where the matrix $\nabla f$ denotes $(\partial_j f_i)_{i,j}$.\\

As already mentioned, we have in mind the application of our criteria to the gas dynamics. 
For this reason, we need to extend the phase space $\mathcal{V}$ to a suitable subset of the boundary of $\mathcal{V}$, to handle the points corresponding to vacuum states. Thus, we introduce as in \cite{Vasseur_Book}:
\[
\mathcal{U}=\{ u\in \bbr^n~|~\exists u_k\in\mathcal{V},\quad \lim_{k\to\infty} u_k =u,\quad \limsup_{k\to\infty} \eta(u_k) <\infty  \},
\]     
and extend the entropy functional $\eta$ on $\mathcal{U}$ by
\[
\eta(\bar{u})=\liminf_{u\in\mathcal{V}, ~u\to\bar{u}} \eta(u).
\]
In $3\times3$ full Euler system, $\mathcal{V}=(0,\infty)\times \bbr\times (0,\infty)$ denotes a set of non-vacuum states of density, momentum and energy, while $\mathcal{U}=\mathcal{V}\cup \{(0,0,0)\}$ includes the vacuum state $(0,0,0)$. In general case, $\mathcal{U}$ is still convex and $\eta$ is convex on $\mathcal{U}$ (See \cite{Vasseur_Book}). We here restrict our study to bounded entropy solutions $u$ to \eqref{main-1}, whose values are in a convex bounded subset $\mathcal{U}_K\subset \mathcal{U}$, on which the functions $f$, $\eta$ and $q$ are continuous.

\subsection{Relative entropy}
For the strictly convex entropy $\eta$ of $\eqref{main-1}$, we define the relative entropy function by
\[
\eta(u|v)=\eta(u)-\eta(v) -\nabla\eta(v)\cdot(u-v),
\]
for any $u\in \mathcal{U}, v\in \mathcal{V}$.\\
Since $\eta$ is convex on $\mathcal{U}$ and strictly convex in $\mathcal{V}$, we have (see \cite{Vasseur_Book})
\[
\eta(u|v)\ge 0,\quad u\in \mathcal{U}, v\in \mathcal{V},
\]
and
\[
\eta(u|v) =0\quad \Longleftrightarrow \quad u=v.
\]
Thus, the relative entropy $\eta(u|v)$ is positive-definite and convex in the first variable $u$. However it looses the symmetry unless $\eta(u)=|u|^2$. Nevertheless the relative entropy is comparable to the square of $L^2$ distance on any bounded subset of $\mathcal{U}$ as follows.
\begin{lemma}\label{lem-L2}
For any bounded set $B\subset \mathcal{U}$ and compact set $\Omega\subset  \mathcal{V}$, there exists $C_1, C_2>0$ depending on $B$ and $\Omega$ such that for any $u\in B$ with $v\in\Omega$,
\[
C_1|u-v|^2 \le \eta(u|v) \le C_2 |u-v|^2.
\]
\end{lemma}
The proof of this lemma can be found in \cite{LV, Vasseur_Book}. Notice that this lemma also holds for all $(u,v)\in \Omega^2$, for any compact set $\Omega\in \mathcal{V}$.\\

As mentioned in Introduction, we are interested in studying the contraction property of a bounded entropy weak solution as any perturbation of admissible Rankine-Hugoniot discontinuity.\\
We say that $u$ is an entropy (weak) solution of \eqref{main} if a weak solution $u$ to \eqref{main} satisfies the entropy inequality
\beq\label{e-ineq}
\partial_t \eta(u) + \partial_x q(u) \le 0,
\eeq
in the sense of distributions. On the other hand, the equality above holds when $u$ is a Lipschitz solution to \eqref{main}.\\

For given $u_l\neq u_r$, we say that $(u_l, u_r, \sigma)$ is an admissible Rankine-Hugoniot discontinuity if there exists $\sigma\in\bbr$ such that
\begin{align}
\begin{aligned} \label{f-q}
&f(u_r)-f(u_l) = \sigma (u_r - u_l)\\
& q(u_r) - q(u_l) \le \sigma (\eta(u_r)- \eta(u_l)). 
\end{aligned}
\end{align}
Equivalently, this means that the discontinuous function $u(t,x)$ defined by
\begin{align*}
\begin{aligned}
u(t,x)=\left\{ \begin{array}{ll}
         u_l & \mbox{if $  x < \sigma t$},\\
         u_r & \mbox{if $ x>\sigma t$},\end{array} \right.
\end{aligned}
\end{align*}
is an entropy weak solution to \eqref{main}.\\

For any constant vector $v\in \bbr^n$, if $u$ is an entropic weak solution of \eqref{main}, then $\eta(u|v)$ is a solution in the sense of distributions to
\beq\label{relative-ineq}
\partial_t \eta(u|v) + \partial_x q(u,v) \le 0,
\eeq
where $q(u,v)$ is the relative entropy flux defined by
\[
q(u,v)=q(u)-q(v) -\nabla\eta(v)\cdot(f(u)-f(v)).
\]
This can be derived directly from \eqref{main} and \eqref{e-ineq}.\\

\subsection{Spatially inhomogeneous pseudo-distance}
For a given weight $a>0$, using the relative entropy, we consider the pseudo-distance $d$ by
\begin{align*}
\begin{aligned}
d(u(t,x),S(t,x))=\left\{ \begin{array}{ll}
         \eta (u(t,x)|u_l) & \mbox{if $  x < \sigma t$},\\
         a\eta (u(t,x)|u_r) & \mbox{if $ x>\sigma t$},\end{array} \right.
\end{aligned}
\end{align*}
where $S(t,x)$ denotes the fixed shock $(u_l,u_r,\sigma)$. This pseudo-distance is spatially inhomogeneous for $a\neq 1$.
Based on this pseudo-distance, it has been shown in \cite{Vasseur-2013} that there exists suitable weight $a>0$ such that
contraction of extremal shocks holds up to Lipschitz shift $\alpha(t)$ in the spatially inhomogeneous pseudo-distance:
\[
 \int_{-\infty}^{\infty} d(u(t,x+\alpha(t)),S(t,x)) dx,
\]
which is equal to
\[
\int_{-\infty}^{\alpha(t)+\sigma t} \eta(u(t,x)|u_l) dx + a\int_{\alpha(t)+\sigma t}^{\infty} \eta(u(t,x)|u_r) dx.
\]
From now on, we study the contraction properties by using this pseudo-distance denoted by \eqref{E_a} as 
\beq\label{E_a}
E_a(u(t),h(t)):= \int_{-\infty}^{h(t)} \eta(u(t,x)|u_l) dx + a\int_{h(t)}^{\infty} \eta(u(t,x)|u_r) dx.
\eeq
This pseudo-distance \eqref{E_a} (determined by the weight $a$) does not depend on $\mathcal{U}_K$. That is, it does not depend on any quantitative property of bounded entropy weak solution perturbed from admissible discontinuities. On the other hand, the shift $h(t)$ depends on the perturbation and is estimated by 
\[
|h^{\prime}(t)|\le C_K,\quad |h(t)-\sigma t|\le C_K \sqrt{t} \|u_0 -S\|_{L^2},
\]
where $C_K$ is a constant depending on $\mathcal{U}_K$.\\

\section{Entropy criteria for contractions}\label{sec-criteria}
In this section, we present a general theory for contraction of admissible discontinuity for any characteristic fields. First of all, we generalize the criteria developed in \cite{Serre-Vasseur} via the spatially inhomogeneous pseudo-distance \eqref{E_a}. We then give the necessary condition for the contraction property in Theorem \ref{thm-cont-2}. Following the heuristic observation of Serre and Vasseur in \cite{Serre-Vasseur}, we consider the generalized conditions for contraction as follows.   
\begin{definition}\label{def-res}
For a positive constant $a>0$, we say that an entropic Rankine-Hugoniot discontinuity $(u_l, u_r)$ is relative entropy stable with respect to weight a (in short, a-RES) if $(u_l, u_r)$ satisfies the following entropy conditions:
\begin{align*}
\begin{aligned} 
&\bullet (\mathcal{H}1):~\mbox{For any}~ u ~\mbox{in}~ \Sigma_a:=\{ u ~|~ \eta(u|u_l) = a\eta(u|u_r) \},\\
&\hspace{1.5cm} D_{sm}(u_{l,r};u):= a q(u,u_r) -q(u,u_l) \le 0.\\
&\bullet (\mathcal{H}2):~\mbox{For any entropic discontinuity} ~(u_-, u_+) ~\mbox{of speed}~ \sigma_{\pm}~ \mbox{satisfying}\\
&\hspace{1.5cm} \eta(u_-|u_l) < a \eta(u_-|u_r) ~\mbox{and}~ \eta(u_+|u_l) > a \eta(u_+|u_r),\\
&\hspace{1.5cm} D_{RH}(u_{l,r},u_{\pm}):= a q(u_+,u_r) - q(u_-,u_l) - \sigma_{\pm} ( a\eta(u_+|u_r) - \eta(u_-|u_l))\le 0.  
\end{aligned}
\end{align*}
As a variant of {\it{a-RES}}, we say that  an entropic Rankine-Hugoniot discontinuity $(u_l, u_r)$ is {\it{strongly relative entropy stable with respect to weight a}} (in short, a-SRES) if $(u_l, u_r)$ satisfies $(\mathcal{H}1)$ and a slightly stronger condition $(\mathcal{H}2^{*})$ than $(\mathcal{H}2)$ as follows:
\begin{align*}
\begin{aligned} 
&\bullet (\mathcal{H}2^*):~\mbox{For any entropic discontinuity} ~(u_-, u_+) ~\mbox{of speed}~ \sigma_{\pm}~ \mbox{such that}\\
&\hspace{1.5cm} \eta(u_-|u_l) - a \eta(u_-|u_r)~\mbox{and}~ a \eta(u_+|u_r)-\eta(u_+|u_l)~\mbox{have the same sign},\\
&\hspace{1.5cm} D_{RH}(u_{l,r},u_{\pm})\le 0.  
\end{aligned}
\end{align*}
\end{definition}

\begin{remark}
1. For the unit weight $a=1$, the meaning of {\it{1-SRES}} is the exactly same as one of the terminology {\it{`RES'}} used in \cite{Serre-Vasseur}. Thus, the above definition of {\it{a-SRES}} is a generalization of {\it{`RES'}}.\\
2. The condition that $\eta(u_-|u_l) - a \eta(u_-|u_r)$ and $a \eta(u_+|u_r)-\eta(u_+|u_l)$ have the same sign has the same meaning as that $u_-$ and $u_+$ are separated by the $(n-1)$-dimensional surface $\Sigma_a$. Note that the convexity and $C^2$-regularity of $\eta$ implies the convexity and $C^2$-regularity of 
\[
R:=\{ u ~|~ \eta(u|u_l) \le a\eta(u|u_r) \},
\]
thus $C^2$-regularity of surface $\Sigma_a = \partial R$.
\end{remark}
\begin{definition}
For a positive constant $a>0$, we say that an entropic Rankine-Hugoniot discontinuity $(u_l, u_r, \sigma)$ satisfies $a$-contraction if for any bounded convex subset $\mathcal{U}_K$ of $\mathcal{U}$, there exists a constance $C_K$ (depending on $\mathcal{U}_K$) such that the following holds true. For any entropy weak solution $u\in \mathcal{U}_K \cap BV_{loc}((0,T)\times\bbr)^n$ (with possibly $T=\infty$) of the system \eqref{main} with initial data $u_0$ satisfying $E(u_0,0)<\infty$, there exists a Lipschitz function $h(t)$ with $h(0)=0$ such that the pseudo distance $E_a(u(t),h(t))$ is non increasing in time, i.e.,
\beq\label{cont}
E_a(u(t),h(t)) \le E_a(u(s),h(s)), \quad \mbox{a.e.}~ t>s>0,
\eeq
moreover, for every $0<t<T$,
\beq\label{h_est}
|h^{\prime}(t)|\le C_K,\quad |h(t)-\sigma t|\le C_K \sqrt{t} \|u_0 -S\|_{L^2},
\eeq
where $S(x)=u_l$ for $x<0$ and $S(x)=u_r$ for $x>0$. 
\end{definition}
\begin{remark}
As already mentioned, the weight $a>0$ defines the pseudo distance $E_a(u(t),h(t))$ for the contraction, which does not depend on any quantitative property of the perturbation, i.e., $\mathcal{U}_K$, but only depends on the system \eqref{main} and the fixed discontinuity $(u_l, u_r, \sigma)$. On the other hand, the boundedness of the set $\mathcal{U}_K$ is for the control of Lipschitz shift $h(t)$ such as \eqref{h_est}. Indeed by Lemma \ref{lem-L2}, $E_a(u(t),h(t))$ is equivalent to $\|u(t,\cdot+h(t)) -S\|_{L^2}$, thus the contraction property \eqref{cont} implies 
\[
\|u(t,\cdot+h(t)) -S\|_{L^2} \le C_K  \|u_0 -S\|_{L^2}.
\]
This induces the estimate \eqref{h_est}. Its proof can be found in \cite{Vasseur-2013}. Therefore, we will not mention about \eqref{h_est} in the sequel.
\end{remark}

\subsection{The {\it{a-SRES}} implies the $a$-contraction} 
We here show that for a given weight $a>0$, the {\it{a-SRES}} is a sufficient condition for the $a$-contraction. This is a generalization of the main theorem (corresponding to $a=1$) in \cite{Serre-Vasseur}.
\begin{theorem}\label{thm-cont-1}
For $a>0$, if the entropic discontinuity $(u_l, u_r)$ is {\it{a-SRES}}, then $(u_l, u_r)$ satisfies $a$-contraction.

\end{theorem}
\begin{proof}
The proof is almost same as that of the main theorem in \cite{Serre-Vasseur}. For the reader's convenience, we give a variant of the proof of Theorem 2.1 in \cite{Serre-Vasseur}. For given $\eps>0$, we define a function $V_{\eps}:\bbr^n \rightarrow \bbr$ by
\begin{align}
\begin{aligned}\label{V-def}
V_{\eps}(u)=\left\{ \begin{array}{ll}
        \displaystyle\frac{[a q(u|u_r) - q(u|u_l) -\eps ]_+}{a \eta(u|u_r) - \eta(u|u_l)} & \mbox{if $u \notin \Sigma_a$},\\
         0 & \mbox{if $u \in \Sigma_a$},\end{array} \right.
\end{aligned}
\end{align}  
where $\Sigma_a=\{ u ~|~ \eta(u|u_l) = a\eta(u|u_r) \}$.\\
Since $(u_{l}, u_{r})$ is {\it{a-SRES}}, $V_{\eps}$ is Lipschitz on $\bbr^n$. Indeed, that is true by the continuity of $q$ and the definition of {\it{a-SRES}}. More precisely, since $a q(u|u_r) - q(u|u_l) -\eps \le -\eps$ for $u\in \Sigma_a$, $V_{\eps} = 0$ on a neighborhood of $\Sigma_a$.\\
We now consider an entropic weak solution $u\in L^{\infty}((0,\infty)\times\bbr)^n \cap BV_{loc}((0,\infty)\times\bbr)^n$ with $E(u_0,0)<\infty$. Then we define an approximated curve $h_{\eps}$ as a solution to the ODE
\begin{align}
\begin{aligned}\label{ode} 
& \dot{h}_{\eps}(t) = V_{\eps} (u(t,h_{\eps}(t))),\\
& h_{\eps}(0)=0,
\end{aligned}
\end{align} 
in the Filippov sense \cite{Filippov}.\\
For solvability of \eqref{ode}, we have the following lemma as in \cite{Serre-Vasseur}:
\begin{lemma}\label{lem-h}
There exists a Lipschitz solution $h_{\eps}$ to \eqref{ode} such that
\begin{align}
\begin{aligned}\label{Set-v} 
& \|\dot{h}_{\eps}\|_{L^{\infty}} \le \|V_{\eps}\|_{L^{\infty}},\\
& \dot{h}_{\eps}(t) \in I(V_{\eps} (u_-) ,V_{\eps} (u_+)),\quad \mbox{a.e.}~t>0,
\end{aligned}
\end{align}
where $u_{\pm}:=u(t,h_{\eps}(t)\pm)$, and $I(a,b)$ denotes the interval with endpoints $a$ and $b$. Moreover, if $u_- \neq u_+$, the $(u_-, u_+, \dot{h}_{\eps})$ is an admissible entropic discontinuity, that is, 
\begin{align}
\begin{aligned}\label{add-ent} 
&f(u_+)-f(u_-) = \dot{h}_{\eps} (u_+ - u_-)\\
& q(u_+) - q(u_-) \le \dot{h}_{\eps} (\eta(u_+)- \eta(u_-)),\quad \mbox{a.e.}~t>0.
\end{aligned}
\end{align}
\end{lemma} 
The proof of this lemma is the exactly same as that in \cite{Serre-Vasseur}, because its proof only need the regularity of $u$ and $V_{\eps}$, but not the definition of $V_{\eps}$ itself. Indeed in \cite{Leger, LV, Serre-Vasseur}, it is shown that \eqref{ode} has a solution satisfying \eqref{Set-v}. We refer to \cite{Serre-Vasseur} for details of its proof.\\
We use \eqref{relative-ineq} to derive that the entropic weak solution $u\in BV_{loc}((0,\infty)\times\bbr)^n$ satisfies
\begin{align*}
\begin{aligned} 
\frac{d}{dt}E_a(u(t),h_{\eps}(t))&= \int_{-\infty}^{h_{\eps}(t)} \partial_t\eta(u(t,x)|u_l) dx + a\int_{h_{\eps}(t)}^{\infty} \partial_t \eta(u(t,x)|u_r) dx\\
&\quad +\dot{h}_{\eps}(t) \Big( \eta(u(t,h(t)-)|u_l) - a \eta(u(t,h(t)+)|u_r) \Big)\\
&\le -\int_{-\infty}^{h_{\eps}(t)} \partial_x q(u(t,x)|u_l) dx - a\int_{h_{\eps}(t)}^{\infty} \partial_x q(u(t,x)|u_r) dx\\
&\quad +\dot{h}_{\eps}(t) \Big( \eta(u(t,h(t)-)|u_l) - a \eta(u(t,h(t)+)|u_r) \Big)\\
&\le aq(u_+|u_r) - q(u_-|u_l) -\dot{h}_{\eps}(t) (a\eta(u_+ | u_r)-\eta(u_-|u_l))\\
& =: D_{\eps},
\end{aligned}
\end{align*}
where $u_{\pm}:=u(t,h_{\eps}(t)\pm)$. Let us show that $D_{\eps} \le 0$.\\
For a.e. $t>0$ such that $u_-=u_+$, by \eqref{Set-v}, we have
\[
\dot{h}_{\eps}(t) = V_{\eps} (u_{\pm}), 
\]
which implies together with \eqref{V-def} that
\begin{align*}
\begin{aligned} 
D_{\eps}&=aq(u_+|u_r) - q(u_-|u_l) -V_{\eps} (u_{\pm}) (a\eta(u_+ | u_r)-\eta(u_-|u_l))\\
& = aq(u_+|u_r) - q(u_-|u_l) -[aq(u_+|u_r) - q(u_-|u_l) - \eps]_+ \\
&\le \eps.
\end{aligned}
\end{align*}
On the other hand, for a.e. $t>0$ such that $u_-\neq u_+$, there are two cases as follows:
\begin{align*}
\begin{aligned} 
&i)~\eta(u_-|u_l) - a \eta(u_-|u_r)~\mbox{and}~ a \eta(u_+|u_r)-\eta(u_+|u_l) ~\mbox{have the same sign,}\\
&ii)~a \eta(u_-|u_r)-\eta(u_-|u_l) ~\mbox{and}~ a \eta(u_+|u_r)-\eta(u_+|u_l) ~\mbox{have the same sign.}
\end{aligned}
\end{align*}
Concerning the first case $i)$, we use the fact that $(u_l, u_r)$ satisfies $(\mathcal{H}2^*)$ and $(u_-, u_+, \dot{h}_{\eps})$ is the entropic discontinuity by \eqref{add-ent}, then it follows from $(\mathcal{H}2^*)$ that
\[
D_{\eps} = D_{RH}(u_{l,r},u_{\pm}) \le 0,
\] 
where note that $\sigma_{\pm}=\dot{h}_{\eps}$.\\
For the second case $ii)$, since $a \eta(u_-|u_r)-\eta(u_-|u_l)$, $a \eta(u_+|u_r)-\eta(u_+|u_l)$, $ V_{\eps} (u_-)$ and $ V_{\eps} (u_+)$ have the same sign, thus $\dot{h}_{\eps} \in I(V_{\eps} (u_-) ,V_{\eps} (u_+))$ also has the same sign. This implies together with \eqref{add-ent} that for both $v=u_-$ and $v=u_+$,
\begin{align*}
\begin{aligned} 
D_{\eps}&\le aq(v|u_r) - q(v|u_l) -\dot{h}_{\eps} (a\eta(v | u_r)-\eta(v|u_l))\\
&=aq(v|u_r) - q(v|u_l) -|\dot{h}_{\eps}| |a\eta(v | u_r)-\eta(v|u_l)|.
\end{aligned}
\end{align*}
If we consider $v$ satisfying $| V_{\eps} (v)|=\mbox{inf}(| V_{\eps} (u_-)|, | V_{\eps} (u_+)|)$, since $|\dot{h}_{\eps}|\ge| V_{\eps} (v)|$ by \eqref{Set-v}, we have
\begin{align*}
\begin{aligned} 
D_{\eps}&\le aq(v|u_r) - q(v|u_l) -| V_{\eps} (v)|  |a\eta(v | u_r)-\eta(v|u_l)|\\
&= aq(v|u_r) - q(v|u_l) - V_{\eps} (v) (a\eta(v | u_r)-\eta(v|u_l))\\
&= aq(v|u_r) - q(v|u_l) - [aq(v|u_r) - q(v|u_l) -\eps]_+\\
&\le \eps.
\end{aligned}
\end{align*}
Therefore, it follows from the estimates above that for a.e. $t>s>0$,
\[
E_a(u(t),h_{\eps}(t)) \le E_a(u(s),h_{\eps}(s))  + (t-s)\eps.
\]
Since $\|V_{\eps}\|_{L^{\infty}}$ is uniformly bounded with respect to $\eps$ thanks to $u\in L^{\infty}((0,\infty)\times\bbr)^n$ and \eqref{V-def}, by \eqref{Set-v}, $\dot{h}_{\eps}$ is also uniformly bounded with respect to $\eps$. Thus up to a subsequence, $h_{\eps}$ uniformly converges to a Lipschitz function $h$. Hence we conclude that a.e. $t>s>0$,
\[
E_a(u(t),h(t)) \le E_a(u(s),h(s)). 
\]
\end{proof}

\subsection{The $a$-contraction implies the {\it{a-RES}}} 
The following theorem says that the {\it{a-RES}} is the necessary condition for the $a$-contraction. 
\begin{theorem}\label{thm-cont-2}
For $a>0$, if the entropic discontinuity $(u_l, u_r)$ satisfies the $a$-contraction, then $(u_l, u_r)$ is the {\it{a-RES}} discontinuity.
\end{theorem}
\begin{proof}
Suppose that the entropic discontinuity $(u_{l}, u_{r})$ is not {\it{a-RES}}, by the definition of {\it{a-RES}}, at least one of $(\mathcal{H}1)$ and $(\mathcal{H}2)$ dose not hold. That is, we assume that one of the following conditions holds.
\begin{align*}
\begin{aligned} 
&\bullet (\sim\mathcal{H}1): ~\exists~ \bar{u} \quad \mbox{such that}\quad \eta(\bar{u}|u_l) = a\eta(\bar{u}|u_r)~\mbox{and}~ D_{sm}(u_{l,r};\bar{u}) > 0.\\
&\bullet (\sim\mathcal{H}2): ~\exists~\mbox{entropic discontinuity}~(u_-,u_+,\sigma_{\pm})\quad \mbox{such that}\\
&\hspace{1.5cm} \eta(u_-|u_l) < a \eta(u_-|u_r),~ \eta(u_+|u_l) > a \eta(u_+|u_r)~\mbox{and}~  D_{RH}(u_{l,r},u_{\pm})>0.
\end{aligned}
\end{align*}
We may show that both cases above provide a contradiction with the contractivity \eqref{cont}.\\   
$\bullet$ {\bf Case of $(\sim\mathcal{H}1)$} : Consider a smooth initial data $u_0$ defined by
\begin{align*}
\begin{aligned}
u_0(x)=\left\{ \begin{array}{ll}
        \bar{u} & \mbox{if $x\in (-R,R)$ for some $R>0$},\\
         u_l & \mbox{if $x \in (-\infty, -2R)$},\\
         u_r & \mbox{if $x \in (2R, \infty)$},\end{array} \right.
\end{aligned}
\end{align*} 
where $\bar{u}$ is a constant vector appeared in $(\sim\mathcal{H}1)$.\\
Then, $E_a(u_0,0)<\infty$ and the system \eqref{main} admits the smooth solution $u$ for some small time $t< T_0$,
which satisfies
\begin{align*}
\begin{aligned}
\frac{d}{dt}E_a(u(t),h(t)) &= \int_{-\infty}^{h(t)} \partial_t\eta(u(t,x)|u_l) dx + a\int_{h(t)}^{\infty} \partial_t \eta(u(t,x)|u_r) dx\\
&\quad +\dot{h}(t) \Big( \eta(u(t,h(t))|u_l) - a \eta(u(t,h(t))|u_r) \Big).
\end{aligned}
\end{align*} 
Since $u$ is smooth for such short time, $u$ satisfies the entropy equality
\[
\partial_t \eta(u) + \partial_x q(u)=0.
\]
Thus, for any constant vector $v\in \bbr^n$, the smooth solution $u$ verifies 
\beq\label{relative-eq}
\partial_t \eta(u|v) + \partial_x q(u|v)=0.
\eeq
We use \eqref{relative-eq} to get
\begin{align*}
\begin{aligned}
\frac{d}{dt}E_a(u(t),h(t)) &= -\int_{-\infty}^{h(t)} \partial_x q(u(t,x)|u_l) dx - a\int_{h(t)}^{\infty} \partial_x q(u(t,x)|u_r) dx\\
&\quad +\dot{h}(t) \Big( \eta(u(t,h(t))|u_l) - a \eta(u(t,h(t))|u_r) \Big)\\
&= a q(u(t,h(t))|u_r) - q(u(t,h(t))|u_l) \\
&\quad +\dot{h}(t) \Big( \eta(u(t,h(t))|u_l) - a \eta(u(t,h(t))|u_r) \Big).
\end{aligned}
\end{align*} 
For any Lipschitz curve $h(t)$ with $h(0)=0$, since we can choose $T_0$ small enough such that
\[
u(t,h(t))=\bar{u}\quad\mbox{for all}~t \in [0,T_0],
\]
it follows from $(\sim\mathcal{H}1)$ that for all $t \in [0,T_0],$
\begin{align*}
\begin{aligned}
\frac{d}{dt}E_a(u(t),h(t)) &= a q(\bar{u}|u_r) - q(\bar{u}|u_l) +\dot{h}(t) \Big( \eta(\bar{u}|u_l) - a \eta(\bar{u}|u_r) \Big)\\
&= a q(\bar{u}|u_r) - q(\bar{u}|u_l) \\
&=D_{sm}(u_{l,r};\bar{u}) > 0,
\end{aligned}
\end{align*} 
which contradicts with \eqref{cont}.\\
$\bullet$ {\bf Case of $(\sim\mathcal{H}2)$} : For the entropic discontinuity $(u_-,u_+,\sigma_{\pm})$ in $(\sim\mathcal{H}2)$, we consider a initial data $u_0$ that is discontinuous at $x=0$ and smooth on $(-\infty,0)$ and $(0,\infty)$, and satisfies that  for some $R>0$,
\begin{align*}
\begin{aligned}
u_0(x)=\left\{ \begin{array}{ll}
         u_- & \mbox{if $x\in (-R,0)$} ,\\
         u_+ & \mbox{if $x\in (0,R)$},\\
         u_l & \mbox{if $x \in (-\infty, -2R)$},\\
         u_r & \mbox{if $x \in (2R, \infty)$},\end{array} \right.
\end{aligned}
\end{align*} 
In order to freeze the shock speed $\sigma_{\pm}$, we consider a new flux $A$ defined by
\beq\label{new-flux}
A(u)=f(u)-\sigma_{\pm} u.
\eeq
Indeed, by Rankine-Hugoniot condition, we have
\[
A(u_+)-A(u_-)= f(u_+)-f(u_-) -\sigma_{\pm} (u_+ -u_-)=0,
\]
which means that the speed of entropic shock $(u_-,u_+)$ to the system 
\beq\label{new-eq}
\partial_t u + \partial_x A(u) = 0
\eeq
is zero. Let $G$ be a $C^2$-entropy flux such that $G^{\prime}=\eta^{\prime} A^{\prime}$. Then, since
\[
G^{\prime}=\eta^{\prime} (f^{\prime}-\sigma_{\pm})=q^{\prime} -\sigma_{\pm}\eta^{\prime},
\]
we have 
\beq\label{G-q}
G=q-\sigma_{\pm}\eta + C\quad \mbox{for some constant}~C.
\eeq
Let us consider a weak entropic solution $w\in L^{\infty}((0,\infty)\times\bbr)^n \cap BV_{loc}((0,\infty)\times\bbr)^n$ to the system \eqref{new-eq} with initial data $u_0$. Then, there exists small time $T_*>0$ such that the weak entropic solution $w$ is smooth on both $(-\infty,0)$ and $(0,\infty)$, and has a shock $(u_-,u_+)$ with zero speed  at $x=0$, and satisfies that for all $t\le T_*$,
\begin{align}\label{u-0}
\begin{aligned}
w(t,x)=\left\{ \begin{array}{ll}
         u_- & \mbox{if $x\in (-R/2,0)$},\\
         u_+ & \mbox{if $x\in (0,R/2)$},\\
          u_l & \mbox{if $x \in (-\infty, -3R)$},\\
         u_r & \mbox{if $x \in (3R, \infty)$}. \end{array} \right. 
\end{aligned}
\end{align} 
Let $h(t)$ be any Lipschitz curve with $h(0)=0$. Then, we can choose $T_*>0$ small enough such that 
\[
|h(t)|<\frac{R}{4}\quad t\in[0,T_*].
\]
By the continuity of $h$, we use the Weierstrass approximation theorem to choose a sequence $(p_n)$ of polynomials such that
\beq\label{weier}
\|p_n-h\|_{L^{\infty}([0,T_*])} ~\rightarrow 0, \quad n\rightarrow \infty,
\eeq
which yields that for sufficiently large $N$,
\beq\label{p-n}
|p_n(t)|<\frac{R}{4},\quad n\ge N,\quad t\in[0,T_*].
\eeq
We first show that
\beq\label{claim-p}
E_a(w(t),p_n(t))\ge E_a(u_0,p_n(0)) + t D_{RH}(u_{l,r},u_{\pm}) ,\quad t\in(0,T_*].
\eeq
Let us begin by noticing the fact that 
\beq\label{w-0}
w(t,x)~\mbox{is constant in time}~ t\in (0,T_*],\quad\mbox{for}~x\in (-\frac{R}{2},\frac{R}{2}),
\eeq 
thus $w(t,w)$ is differentiable with respect to $t\in (0,T_*]$ for all $x\in \bbr$. This yields
\begin{align*}
\begin{aligned}
\frac{d}{dt}E_a(w(t),p_n(t)) &= \underbrace{\int_{-\infty}^{p_n(t)} \partial_t\eta(w(t,x)|u_l) dx}_{I_1} + \underbrace{a\int_{p_n(t)}^{\infty} \partial_t \eta(w(t,x)|u_r) dx}_{I_2}\\
&\quad +\underbrace{\dot{p}_n(t) \Big( \eta(w(t,p_n(t)-)|u_l) - a \eta(w(t,p_n(t)+)|u_r) \Big)}_{I_3}.
\end{aligned}
\end{align*}
{\bf (Case of non-constant polynomial $p_n$)} : If the polynomial $p_n$ is not constant for some $n\ge N$, $p_n$ has a finite number of zeros as $t_1, t_2\cdots,t_k$ with
\[
0\le t_1<t_2<\cdots<t_k\le T_*.
\]
Then $p_n$ is either negative or positive on each interval $(t_i,t_{i+1})$, where $i$ belongs to one of the following classes
\begin{align*}
\begin{aligned}
&1\le i\le k-1 \quad \mbox{if $t_1=0$ and $t_k=T_*$},\\
&0\le i\le k-1 \quad \mbox{if $t_1>0$ and $t_k=T_*$ where $t_0:=0$},\\
&1\le i\le k \quad \mbox{if $t_1=0$ and $t_k<T_*$ where $t_{k+1}:=T_*$},\\
&0\le i\le k \quad \mbox{if $t_1>0$ and $t_k<T_*$}.
\end{aligned}
\end{align*} 
First of all, let us consider the case that $p_n(t)<0$ for all $t\in (t_l,t_{l+1})$ for some $l$. Then the solution $w$ is smooth on $(0,T_*]\times(-\infty, p_n(t)]$. Thus, we apply \eqref{relative-eq} to the system \eqref{new-eq} to get
\[
I_1=-\int_{-\infty}^{p_n(t)} \partial_x G(w(t,x)|u_l) dx.
\]
Since by \eqref{u-0} and \eqref{p-n},
\beq\label{w-l}
w(t,p_n(t))=u_-\quad\mbox{ for all} ~t\in (t_l,t_{l+1}), 
\eeq
we have
\[
I_1= -G(u_- |u_l).
\]
Since $w$ is discontinuous at $x=0$ for $0<t\le T_*$, we rewrite $I_2$ as
\[
I_2 =a\int_{p_n(t)}^{\frac{R}{4}} \partial_t \eta(w(t,x)|u_r) dx +a\int_{\frac{R}{4}}^{\infty} \partial_t \eta(w(t,x)|u_r) dx.
\]
By \eqref{w-0}, we have 
\[
\partial_t \eta(w(t,x)|u_r)=0,\quad (t,x)\in (0,T_*]\times [p_n(t),\frac{R}{4}].
\]
Since $w$ is smooth on $(0,T_*]\times[\frac{R}{4},\infty)$ and $w(t,\frac{R}{4})=u_+$ for $0<t\le T_*$, we have
\[
I_2 =-a\int_{\frac{R}{4}}^{\infty}\partial_x G(w(t,x)|u_r) dx = a G(u_+| u_r).
\]
For $I_3$, we use \eqref{w-l} to get
\[
I_3 =\dot{p}_n(t) \Big( \eta(u_-|u_l) - a \eta(u_-|u_r) \Big).
\]
Thus, we have shown
\begin{align*}
\begin{aligned}
\frac{d}{dt}E_a(w(t),p_n(t)) = a G(u_+| u_r) -G(u_- |u_l)+\dot{p}_n(t) \Big( \eta(u_-|u_l) - a \eta(u_-|u_r) \Big),\quad t\in (t_l,t_{l+1}).
\end{aligned}
\end{align*}
We here use \eqref{G-q} to reduce 
\begin{align*}
\begin{aligned}
&a G(u_+| u_r) -G(u_- |u_l)\\
&\quad = a\Big[ q(u_+) -\sigma_{\pm} \eta(u_+) - q(u_r) +\sigma_{\pm} \eta(u_r) - d\eta (u_r) (f(u_+)-\sigma_{\pm} u_+ -f(u_r) -\sigma_{\pm} u_r )   \Big]\\
&\qquad   -q(u_-) +\sigma_{\pm} \eta(u_-) + q(u_l) -\sigma_{\pm} \eta(u_l) + d\eta (u_l) (f(u_-)-\sigma_{\pm} u_- -f(u_l) -\sigma_{\pm} u_l ) \\
&\quad = a q(u_+| u_r) -q(u_- |u_l) - \sigma_{\pm} \Big(a\eta(u_+| u_r) -\eta(u_- |u_l)\Big)\\
&\quad =D_{RH}(u_{l,r},u_{\pm}),
\end{aligned}
\end{align*} 
which provides
\begin{align*}
\begin{aligned}
\frac{d}{dt}E_a(w(t),p_n(t)) = D_{RH}(u_{l,r},u_{\pm})+\dot{p}_n(t) \Big( \eta(u_-|u_l) - a \eta(u_-|u_r) \Big),\quad t\in (t_l,t_{l+1}).
\end{aligned}
\end{align*}
For any $\tau,t \in (t_l,t_{l+1})$ with $\tau<t$, integrating it over $[\tau,t]$, we have
\begin{align*}
\begin{aligned}
E_a(w(t),{p}_n(t))&=E_a(w(\tau),{p}_n(\tau)) +(t-\tau) D_{RH}(u_{l,r},u_{\pm})\\
&\quad+({p}_n(t)-{p}_n(\tau)) \Big( \eta(u_-|u_l) - a \eta(u_-|u_r) \Big).
\end{aligned}
\end{align*}
Since $E_a(w(t),p_n(t))$ is continuous in time $t$ and ${p}_n(t_l)=0$, taking $\tau\to t_l$, we have
\begin{align*}
\begin{aligned}
E_a(w(t),p_n(t))&=E_a(w(t_{l}),p_n(t_{l})) +(t-t_{l}) D_{RH}(u_{l,r},u_{\pm})\\
&\quad+{p}_n(t) \Big( \eta(u_-|u_l) - a \eta(u_-|u_r) \Big).
\end{aligned}
\end{align*}
Since ${p}_n(t)<0$ for all $t\in (t_l,t_{l+1})$, we use the assumption $(\sim\mathcal{H}2)$ to get 
\begin{align}
\begin{aligned}\label{L-th}
E_a(w(t),p_n(t))&\ge E_a(w(t_{l}),p_n(t_{l})) +(t-t_{l}) D_{RH}(u_{l,r},u_{\pm}),\quad t\in (t_l,t_{l+1}],
\end{aligned}
\end{align}
where note that this inequality also holds at $t=t_{l+1}$ by the time-continuity of $E_a(w(t),p_n(t))$.

On the other hand, for the case where $p_n(t)>0$ for all $t\in (t_m,t_{m+1})$ for some $m$, we follow the same argument as above. More precisely for all $t\in (t_m,t_{m+1})$, we get
\begin{align*}
\begin{aligned}
I_1 &=\int_{-\infty}^{-\frac{R}{4}} \partial_t \eta(w(t,x)|u_l) dx + \int_{-\frac{R}{4}}^{p_n(t)} \partial_t \eta(w(t,x)|u_l) dx\\
&=-\int_{-\infty}^{-\frac{R}{4}}  \partial_x G(w(t,x)|u_l) dx = - G(u_-| u_l),\\
I_2&=-a\int_{p_n(t)}^{\infty} \partial_x G(w(t,x)|u_r) dx = a G(u_+|u_r),\\
I_3 &=\dot{p}_n(t) \Big( \eta(u_+|u_l) - a \eta(u_+|u_r) \Big),
\end{aligned}
\end{align*}
which provides that for all $t\in (t_m,t_{m+1})$,
\begin{align*}
\begin{aligned}
E_a(w(t),p_n(t))&=E_a(w(t_{m}),p_n(t_{m})) +(t-t_{m}) D_{RH}(u_{l,r},u_{\pm})\\
&\quad+{p}_n(t) \Big( \eta(u_+|u_l) - a \eta(u_+|u_r) \Big).
\end{aligned}
\end{align*}
Since ${p}_n(t)>0$ for all $t\in (t_m,t_{m+1})$, we use the assumption $(\sim\mathcal{H}2)$ to get 
\begin{align}
\begin{aligned}\label{m-th}
E_a(w(t),p_n(t))&\ge E_a(w(t_{m}),p_n(t_{m})) +(t-t_{m}) D_{RH}(u_{l,r},u_{\pm}),\quad t\in (t_m,t_{m+1}].
\end{aligned}
\end{align}
Therefore, combining \eqref{L-th} and \eqref{m-th}, we can conclude \eqref{claim-p}.\\
{\bf (Case of constant polynomial $p_n$)} : If polynomial $p_n$ is constant for some $n\ge N$, using \eqref{p-n}, we have
\[
p_n(t)=p_n(0)\in (-\frac{R}{4},\frac{R}{4})\quad \mbox{for all}~ t\in[0,T_*].
\]
First of all, for the case of $p_n(0)\neq 0$, the claim \eqref{claim-p} follows directly from the previous arguments.
If $p_n(0) = 0$, we also combine the previous arguments to have
\begin{align*}
\begin{aligned}
I_1 &=\int_{-\infty}^{-\frac{R}{4}} \partial_t \eta(w(t,x)|u_l) dx + \int_{-\frac{R}{4}}^{0} \partial_t \eta(w(t,x)|u_l) dx\\
&=-\int_{-\infty}^{-\frac{R}{4}}  \partial_x G(w(t,x)|u_l) dx = - G(u_-| u_l),\\
I_2&=a\int_{0}^{\frac{R}{4}} \partial_t \eta(w(t,x)|u_r) dx + a\int_{\frac{R}{4}}^{\infty} \partial_t \eta(w(t,x)|u_r) dx\\
&=-a\int_{\frac{R}{4}}^{\infty}  \partial_x G(w(t,x)|u_r) dx = a G(u_+|u_r),\\
I_3 &=0,
\end{aligned}
\end{align*}
which provides \eqref{claim-p}.

Since $t D_{RH}(u_{l,r},u_{\pm})$ in \eqref{claim-p} independent of $n$ and $(\sim\mathcal{H}2)$ implies
\[
t D_{RH}(u_{l,r},u_{\pm})>0\quad \mbox{for}~t>0,
\]
using \eqref{weier}, \eqref{claim-p} and $w\in L^{\infty}([0,\infty)\times\bbr)$, we show
\beq\label{final-0}
E_a(w(t),h(t)) > E_a(u_0,0), \quad t\in(0,T_*].
\eeq
Indeed, since
\begin{align*}
\begin{aligned}
&|E_a(w(t),h(t))-E_a(w(t),p_n(t))|\\
&\quad=\Big| \int_{p_n(t)}^{h(t)} \eta(w(t,x)|u_l) dx + a\int_{h(t)}^{p_n(t)} \eta(w(t,x)|u_r) dx \Big|\\
&\quad \le \|p_n-h\|_{L^{\infty}([0,T_*])} \Big( \|\eta(w|u_l)\|_{L^{\infty}([0,\infty)\times\bbr)} + a \|\eta(w|u_r)\|_{L^{\infty}([0,\infty)\times\bbr)} \Big)\\
&\quad \rightarrow 0\quad \mbox{as} ~n\to \infty,
\end{aligned}
\end{align*}
we have
\[
E_a(w(t),p_n(t)) \rightarrow  E_a(w(t),h(t))\quad\mbox{uniformly in}~t\in[0,T_*],
\]
which implies that for $t\in(0,T_*]$,
\begin{align*}
\begin{aligned}
E_a(w(t),h(t))\ge E_a(u_0,0) + \frac{t}{2} D_{RH}(u_{l,r},u_{\pm})>E_a(u_0,0).
\end{aligned}
\end{align*}
Since the Lipschitz shift $h$ is arbitrary, the weak entropic solution $u$ of \eqref{main} with $u_0$ also satisfies
\[
E_a(u(t),h(t))>E_a(u_0,0),\quad 0<t\le T_*.
\] 
which provides the contradiction with \eqref{cont}.\\
\end{proof}

\section{On $a$-contractions for extremal shocks}\label{cont-shock}
In this section, we are going to show that extremal shocks (i.e. $1$-shock or $n$-shock) is {\it{a-SRES}} for some $a$, which implies that they satisfies $a$-contraction by Theorem \ref{thm-cont-1}. Even though the $a$-contraction for extremal shocks has been shown by Vasseur in \cite{Vasseur-2013}, we intend to here give an alternate proof as a direct application of the criteria built in Theorem \ref{thm-cont-1}. This result goes beyond the known results valid in the class of $BV$ solutions under small perturbation in $BV$. In the case of small perturbation in $L^{\infty}\cap BV$, Bressan, Crasta and Piccoli in \cite {Bressan1} developed a powerful theory of $L^1$ stability for entropy solution obtained by either the Glimm scheme or the wave front-tracking method. The theory also works in some cases for small perturbation of large shock (See \cite{Lewicka, Bressan3}). On the other hand, Chen, Frid and Li in \cite{Chen1} use the relative entropy to establish the uniqueness and stability of solutions to the Riemann problem for the $3\times 3$ Euler system in a large $L^1\cap L^{\infty}\cap BV_{loc}$ perturbation (See also \cite{CF1, Frid1}).\\

\subsection{Hypotheses}
We suppose the same hypotheses for the system \eqref{main} as in \cite{Vasseur-2013}, which are especially applied to the isentropic Euler system and full Euler system. (See \cite{LV, Vasseur-2013})\\

The following hypotheses are related to the $1$-shock and other entropic discontinuities.
\bigskip

\begin{itemize}
\item $(\mathcal{H})$ :  
For any fixed $u_l\in \mathcal{V}$ and $1 \le i\le n$, there exists a neighborhood $B\subset \mathcal{V}$ of $u_l$ such that for any $u\in B$, there is a $i$-th Hugoniot curve $S_{u}^i(s) \in \mathcal{U}$ defined on an interval $[0,s_u)$ (possibly $s_u=\infty$), such that 
$S_{u}^i(0)=u$ and the Rankine-Hugoniot condition:
\[f(S_{u}^i(s)) -f(u) = \sigma_{u}^i(s) (S_{u}^i(s) - u),\]
where $\sigma_{u}^i(s)$ is a velocity function. Here, $u \rightarrow s_{u}$ is Lipschitz on $\mathcal{U}$, $(s,u) \rightarrow S^i_u(s)$ and 
$(s,u) \rightarrow \sigma^i_u(s)$ are both $C^1$ on $\{ (s,u) ~|~ s\in[0,s_u),~u\in \mathcal{U} \}$,
and the following conditions are satisfied.
\begin{align}
\begin{aligned}\label{Liu-0}
\mbox{(a)}~ \frac{d}{ds}\sigma_{u}^1(s) < 0,\quad\sigma_{u}^1(0)=\lambda_1(u),\quad \mbox{(b)}~ \frac{d}{ds} \eta(u|S_u^1(s)) >0,\quad \mbox{for all}~s>0,
\end{aligned}
\end{align} 
\beq\label{Lax-0}
\lambda_{1} (S_{u}^1(s)) < \sigma_{u}^1(s)< \lambda_{1} (u) \le\sigma_{u}^i(s),\quad 2\le i\le n,\quad s>0, 
\eeq
and
\begin{align}
\begin{aligned}\label{other-con}
\frac{d}{ds}\sigma_{u}^i(s) \le 0,\quad \sigma_{u}^i(0)=\lambda_i(u)\quad \mbox{for $2\le i\le n$}.
\end{aligned}
\end{align} 
\end{itemize}
Regarding the hypotheses for the $n$-shock, we just replace \eqref{Liu-0} and \eqref{Lax-0} by \eqref{Liu-0-1} and \eqref{Lax-0-1} as follows.

\begin{itemize}
\item $(\mathcal{H}^*)$ :  
For any fixed $u_r\in \mathcal{V}$ and $1 \le i\le n$, there exists a neighborhood $B\subset \mathcal{V}$ of $u_r$ such that for any $u\in B$, there is a $i$-th Hugoniot curve $S_{u}^i(s) \in \mathcal{U}$ defined on an interval $[0,s_u)$ (possibly $s_u=\infty$), such that 
$S_{u}^i(0)=u$ and the Rankine-Hugoniot condition:
\[f(S_{u}^i(s)) -f(u) = \sigma_{u}^i(s) (S_{u}^i(s) - u),\]
where $\sigma_{u}^i(s)$ is a velocity function. Here, $u \rightarrow s_{u}$ is Lipschitz on $\mathcal{U}$, $(s,u) \rightarrow S^i_u(s)$ and 
$(s,u) \rightarrow \sigma^i_u(s)$ are both $C^1$ on $\{ (s,u) ~|~ s\in[0,s_u),~u\in \mathcal{U} \}$,
and the following conditions are satisfied.
\begin{align}
\begin{aligned}\label{Liu-0-1}
 \frac{d}{ds}\sigma_{u}^n(s) > 0,\quad  \frac{d}{ds} \eta(u|S_u^n(s)) >0,\quad \mbox{for all}~s>0,
\end{aligned}
\end{align} 
\beq\label{Lax-0-1}
\lambda_{n} (S_{u}^n(s)) > \sigma_{u}^n(s)> \lambda_{n} (u) \ge\sigma_{u}^i(s),\quad 1\le i\le n-1,\quad s>0, 
\eeq
and
\begin{align*}
\begin{aligned}
\frac{d}{ds}\sigma_{u}^i(s) \le 0,\quad \sigma_{u}^i(0)=\lambda_i(u)\quad \mbox{for $1\le i\le n-1$}.
\end{aligned}
\end{align*} 
\end{itemize}
\begin{remark}\label{remark-dual}
1. Note that a system \eqref{main} verifies the hypotheses $(\mathcal{H})$ relative to $u_l\in \mathcal{V}$ if and only if the system 
\[
\partial_t u -\partial_x f(u)=0
\]
verifies $(\mathcal{H}^*)$ relative to $u_r\in \mathcal{V}$. In other words, $(\mathcal{H})$ and $(\mathcal{H}^*)$ are dual in this way. Thus it is enough to show the $a$-contraction for 1-shocks, because the result of $n$-shock is obtained by applying the case of $1$-shock to $\tilde{u}(t,x)= u(t,-x)$, which is also entropic solution to \eqref{main}. From now on, we will restrict our arguments to the case of a 1-shock.\\
2. The assumptions (a) in \eqref{Liu-0}, \eqref{Lax-0} and \eqref{other-con} are just due to Liu and Lax entropy conditions. The only additional requirement is (b) in \eqref{Liu-0}, which is a condition on the growth of the shock along $S_u^1(s)$ measured through the pseudo-distance. This condition arises naturally in the study of admissibility criteria for systems of conservation laws. In particular, it ensures that Liu admissible shocks are entropic even for moderate to strong shocks. Indeed, this fact follows
from the important formula : (See also \cite{Dafermos4, Lax, Ruggeri})
\[
q(S_u^1(s))   -  q(u) = \sigma_u^1(s) ( \eta(S_u^1(s)- \eta(u)) + \int_0^s \frac{d}{dt}\sigma^{1}_u(t)\eta(u|S_u^1(t))dt.
\]
In \cite{BFZ}, Barker, Freist\"uhler, and Zumbrun showed that the stability (and so the contraction as well) fails to hold for the full Euler system if hypothesis (b) in \eqref{Liu-0} is replaced by 
\[
 \frac{d}{ds} \eta(S_u^1(s)) >0, \quad s>0.
\]
This shows that the strength of the shock is better measured by the relative entropy rather
than the entropy itself. \\
3. In \cite{TZ}, Texier and Zumbrun showed that the hypotheses $(\mathcal{H})$ implies the Lopatinski condition of Majda.
\end{remark}

\subsection{Structural lemmas}
We first present the following structural Lemmas treated in \cite{Vasseur-2013}. The first lemma provides a kind of triangle inequality for the pseudo metric induced by $\eta(\cdot|\cdot)$ and its analogous inequalities, which are useful tools in the following proofs.
\begin{lemma}\label{lem-metric}
For any $u, v, w \in \bbr^n$, we have
\beq\label{eta-m}
\eta(u|w)+\eta(w|v) = \eta(u|v) + (\nabla\eta(w)-\nabla\eta(v)) \cdot (w-u),
\eeq
and 
\beq\label{q-m}
q(u|w)+q(w|v)= q(u|v) + (\nabla\eta(w)-\nabla\eta(v)) \cdot (f(w)-f(u)).
\eeq
Therefore, for any $\sigma\in\bbr$,
\begin{align}
\begin{aligned}\label{metric}
q(u|v) - \sigma \eta(u|v) &=  \Big( q(u|w) - \sigma \eta(u|w) \Big) + \Big(q(w|v) - \sigma \eta(w|v)\Big)\\
&\quad -  (\nabla\eta(w)-\nabla\eta(v)) \cdot \Big(f(w)-f(u) -\sigma (w-u) \Big)
\end{aligned}
\end{align} 
\end{lemma}
\begin{proof}
The proof follows directly from the definition of the relative entropy $\eta(\cdot|\cdot)$ and its flux $q(\cdot|\cdot)$. Indeed, the following computations hold: 
\begin{align*}
\begin{aligned}
\eta(u|w)+\eta(w|v)&= \Big( \eta({u})-\eta(w) -\nabla\eta(w)\cdot(u-w)\Big)\\
&\quad + \Big(\eta(w)-\eta(v) -\nabla\eta(v)\cdot(w-v) \Big) \\
&=  \eta(u)-\eta(v) -\nabla\eta(v)\cdot(u-v) + (\nabla\eta(w)-\nabla\eta(v)) \cdot (w-u) \\
&= \eta(u|v) + (\nabla\eta(w)-\nabla\eta(v)) \cdot (w-u),
\end{aligned}
\end{align*} 
\begin{align*}
\begin{aligned}
q(u|w)+q(w|v)&= \Big( q({u})-q(w) -\nabla\eta(w)\cdot(f(u)-f(w))\Big)\\
&\quad + \Big(q(w)-q(v) -\nabla\eta(v)\cdot(f(w)-f(v)) \Big) \\
&=  q(u)-q(v) -\nabla\eta(v)\cdot(f(u)-f(v)) + (\nabla\eta(w)-\nabla\eta(v)) \cdot (f(w)-f(u)) \\
&= q(u|v) + (\nabla\eta(w)-\nabla\eta(v)) \cdot (f(w)-f(u)).
\end{aligned}
\end{align*} 
\end{proof}

The following lemma gives an explicit formula concerning the entropy lost at an entropic discontinuity $(u,S^i_{u}(s))$ for any $i$-family.  
\begin{lemma}\label{lem-v1}
For any Rankine-Hugoniot discontinuity $(u,S^i_u(s),\sigma_u^i(s))$ and any vector $v$, we have
\beq\label{vasseur}
q(S_u^i(s),v)   -  \sigma_u^i(s)  \eta(S_u^i(s)|v) =  q(u,v)   -  \sigma_u^i(s)  \eta(u|v) + \int_0^s \frac{d}{dt}\sigma^{i}_u(t)\eta(u|S_u^i(t))dt.
\eeq
Therefore, for any $s\ge 0$, $s_0>0$, we have
\beq\label{vasseur-1}
q(S_u^i(s),S_u^i(s_0))   -  \sigma_u^i(s)  \eta(S_u^i(s)|S_u^i(s_0)) = \int^s_{s_0} \frac{d}{dt}\sigma^i_u(t)\Big( \eta(u|S_u^i(t))-\eta(u|S_u^i(s_0)) \Big) dt. 
\eeq
In particular, for any $u\in B$ as in hypothesis $(\mathcal{H})$, there exists $\delta\in (0,\frac{s_0}{2})$ and $k>0$ such that
\begin{align}
\begin{aligned}\label{delta-k}
&q(S_u^1(s),S_u^1(s_0))   -  \sigma_u^1(s)  \eta(S_u^1(s)|S_u^1(s_0)) \le -k |\sigma_u^1(s)- \sigma_u^1(s_0)|^2,\quad\mbox{for}~ |s-s_0|< \delta,\\
&q(S_u^1(s),S_u^1(s_0))   -  \sigma_u^1(s)  \eta(S_u^1(s)|S_u^1(s_0)) \le -k |\sigma_u^1(s)- \sigma_u^1(s_0)|,\quad\mbox{for}~ |s-s_0|\ge \delta.
\end{aligned}
\end{align}
\end{lemma}
 The proof of this lemma can be found in \cite{Vasseur-2013}. We refer to the work of Lax \cite{Lax} for the estimate \eqref{vasseur}. And the estimates \eqref{vasseur-1} and \eqref{delta-k} are variations on a crucial lemma of DiPerna \cite{DiPerna}. Note that the discontinuity $(u_l,S_u^i(s),\sigma_u^i(s))$ in \eqref{vasseur} and \eqref{vasseur-1} need not be extremal family (i.e. $1$-family or $n$-family) from the proof of the relation \eqref{vasseur} in \cite{Vasseur-2013}, which is obtained directly from the Rankine-Hugoniot condition, moreover, \eqref{vasseur-1} is obtained by using \eqref{vasseur} twice. We give the proof of \eqref{delta-k} in the Appendix for the reader's convenience.\\

For the proof of Theorem \ref{thm-ext}, we present the following lemma that states both Lemma 5 and Proposition 2 in \cite{Vasseur-2013} as a slightly improved version. In fact, Proposition 2 in \cite{Vasseur-2013} says that \eqref{lem-rh} holds for all $s>0$ satisfying $\sigma_{u}(s)\le \sigma_0$, which means that \eqref{lem-rh} holds for any strong shock $(u,S_{u}^1(s), \sigma_{u}^1(s))$ thanks to the assumption $(\mathcal{H})$ with \eqref{Liu-0}. It turns out that the constraint $\sigma_{u}(s)\le \sigma_0$ can be removed in the following lemma. 
\begin{lemma}\label{lem-v2}
Let $(u_l,u_r, \sigma_{l,r})$ be an entropic $1$-shock such that $u_r=S_{u_l}^1(s_0)$, $\sigma_{l,r}=\sigma_{u_l}^1(s_0)$ for some $s_0>0$, and the corresponding conditions in \eqref{Liu-0} are satisfied.
Then, there exists $\sigma_0\in (\sigma_{l,r}, \lambda_1(u_l))$, $\eps_0>0$, $\beta>0$ and $a_*>0$ verifying the following properties: 
\begin{itemize}
\item For any $u\in B_{\eps_0}(u_l)$,
\begin{align}
\begin{aligned}\label{lem-5}
&\sigma_0 \le \lambda_1(u),\\
&-q(u,u_l) +\sigma_0 \eta(u|u_l) \le -\beta  \eta(u|u_l),\\
&q(u,u_r) - \sigma_0 \eta(u|u_r) \le -\beta  \eta(u|u_r).
\end{aligned}
\end{align}
\item
For any $0<a<a_*$, 
\beq\label{domain-r}
\mbox{the ball $B_{\eps_0}(u_l)$ contains the convex set}~{R}_a:=\{ u ~|~ \eta(u|u_l) \le a\eta(u|u_r) \},
\eeq
and for any $u\in {R}_a$,
\beq\label{lem-rh}
a \Big(q(S_{u}^1(s),u_r) - \sigma_{u}^1(s) \eta(S_{u}^1(s)|u_r) \Big) - q(u|u_l) +\sigma_{u}^1(s) \eta(u|u_l) \le 0,\quad s>0.
\eeq
\end{itemize}

\end{lemma}
\begin{proof}
{\bf Step A) proof for \eqref{lem-5}} : 
We begin by using \eqref{vasseur-1} with $u=u_l$ and $s=0$ to get 
\[
q(u_l,u_r) - \lambda_1(u_l) \eta(u_l|u_r)<0,
\]
where we have used the hypothesis \eqref{Liu-0}. Since this inequality is strict, we can choose $\sigma_0$ sufficiently close to $\lambda_1(u_l)$ with $\sigma_0\in (\sigma_{l,r}, \lambda_1(u_l))$ and $\beta>0$ sufficiently small such that
\[
q(u_l,u_r) - (\sigma_0-\beta) \eta(u_l|u_r)<0,
\]
which gives
\[
q(u_l,u_r) - \sigma_0\eta(u_l|u_r)<-\beta\eta(u_l|u_r).
\]
Then, using the continuity of $q(\cdot, u_r)$, $\eta(\cdot|u_r)$ and $\lambda_1(\cdot)$, we can choose $\eps_0$ sufficiently small such that for all $u\in B_{\eps_0}(u_l)$,
\[
q(u,u_r) - \sigma_0\eta(u|u_r)<-\beta\eta(u|u_r),\quad\mbox{and}\quad \sigma_0 \le \lambda_1(u).
\]
On the other hand, we use Taylor expansion at $u_l$ to get
\[
-q(u,u_l) +\sigma_0 \eta(u|u_l)= (u -u_l)^{T} \nabla^2 \eta(u_l) (\sigma_0 I -\nabla f(u_l)) (u -u_l) + \mathcal{O}(|u-u_l|^3).
\]
Here, since the entropy $\eta$ is strictly convex, $\nabla^2\eta(u_l)$ is symmetric and strictly positive, and $ \nabla^2 \eta(u_l)\nabla f(u_l))$ is symmetric. Thus those matrices are diagonalizable in the same basis, which gives
\[ 
\nabla^2 \eta(u_l)\nabla f(u_l)) \ge \lambda_1(u_l) \nabla^2 \eta(u_l).
\]
This and $\lambda_1(u_l)>\sigma_0$ imply that for all $u\in B_{\eps_0}(u_l)$,
\begin{align*}
\begin{aligned}
-q(u,u_l) +\sigma_0 \eta(u|u_l)&\le -(\lambda_1(u_l)-\sigma_0) (u -u_l)^{T} \nabla^2 \eta(u_l)  (u -u_l) + \mathcal{O}(|u-u_l|^3)\\
&\le -(\lambda_1(u_l)-\sigma_0)\eta(u|u_l)+ \mathcal{O}(|u-u_l|^3)\\
&\le -\frac{\lambda_1(u_l)-\sigma_0}{2}\eta(u|u_l),
\end{aligned}
\end{align*} 
where we have used the smallness of $\eps_0$ in the last inequality. \\
{\bf Step B) proof for \eqref{domain-r}} : Notice that
\begin{align}
\begin{aligned}\label{a-1}
&\mbox{for}~a<1,\quad\eta(u|u_l) \le a\eta(u|u_r) \quad\mbox{is equivalent to}\\
&\eta(u) \le \frac{1}{1-a} (\eta(u_l)-a\eta(u_r) -\nabla \eta(u_l)\cdot u_l + a\nabla \eta(u_r)\cdot u_r +(\nabla \eta(u_l)-a\nabla \eta(u_r))\cdot u),
\end{aligned} 
\end{align}
where the right hand side of the second inequality above is linear in $u$. Thus, the convexity of $\eta$ implies the convexity of ${R}_a=\{u~|~ \eta(u|u_l) \le a\eta(u|u_r)  \}$. We take $a_*<\frac{1}{2}$ to rewrite the second inequality above as
\begin{align*}
\begin{aligned}
\eta(u|u_l) &\le \frac{a}{1-a} (\eta(u_l)-\eta(u_r) -\nabla \eta(u_l)\cdot u_l + \nabla \eta(u_r)\cdot u_r +(\nabla \eta(u_l)-\nabla \eta(u_r))\cdot u)\\
&\le Ca(1+|u|),\quad\mbox{for all} ~0<a<a_*.
\end{aligned} 
\end{align*}
This yields together with Lemma \ref{lem-L2} that for all $u\in R_a\cap B_{\eps_0}(u_l)$,
\[
|u-u_l|^2 \le Ca(1+|u|)\le C_*a.
\]
Taking $a_*$ sufficiently small as $a_*<\frac{\eps_0^2}{2 C_*}$ such that for any $a<a_*$ and $u\in R_a\cap B_{\eps_0}(u_l)$,
\[
|u-u_l|^2 \le C_*a < \frac{\eps_0^2}{2},
\]
which implies that $B_{\eps_0}(u_l)$ strictly contains $R_a\cap B_{\eps_0}(u_l)$. Since $R_a$ is convex, so connected, thus we have
\[
R_a = R_a\cap B_{\eps_0}(u_l).
\]
Therefore, for any $a<a_*$,
\[
R_a \subset  B_{\eps_0}(u_l).
\]

{\bf Step C) proof of \eqref{lem-rh}} :  
First of all, we take $\sigma_0$ closer to $\lambda_1(u_l)$ and $\eps_0$ smaller than those chosen in Step A, such that for all $u\in  B_{\eps_0}(u_l)$
\beq\label{sigma-con}
\sigma^1_{u}(s)\le \sigma_0,\quad s\ge \frac{s_0}{2},
\eeq
and
\beq\label{later-c}
\lambda_1(u)-\sigma_0<\frac{k}{4 C_1} |\sigma_{u}^1(\frac{s_0}{2})-\sigma_{u}^1(s_0)|,
\eeq
where $k>0$ is the constant as in \eqref{delta-k} and $C_1$ is appeared in \eqref{C_1}.\\  
Indeed, \eqref{sigma-con} can be justified thanks to the assumption $(\mathcal{H})$, in which $(s,u) \rightarrow \sigma_u(s)$ is $C^1$-function, and $\frac{d}{ds}\sigma_{u}^1(s) < 0$ with $\sigma_{u}^1(0)=\lambda_1(u)$.\\
Let us first show \eqref{lem-rh} for all $s\ge \frac{s_0}{2}$. We use \eqref{metric} to get
\begin{align*}
\begin{aligned}
&q(S_{u}^1(s),u_r) - \sigma_{u}^1(s) \eta(S_{u}^1(s)|u_r) \\
&\quad = -\Big(q(u_r,S_{u}^1(s_0)) - \sigma_{u}^1(s) \eta(u_r|S_{u}^1(s_0))\Big)+\Big(q(S_{u}^1(s),S_{u}^1(s_0)) - \sigma_{u}^1(s) \eta(S_{u}^1(s)|S_{u}^1(s_0))\Big)\\
&\qquad + \Big(\nabla\eta (u_r) -\nabla\eta (S_{u}^1(s_0))  \Big) \Big(f(u_r)-f(S_{u}^1(s)) -\sigma_u^1 (s) (u_r-S_{u}^1(s)) \Big),
\end{aligned}
\end{align*} 
where $u_r=S_{u_l}^1(s_0)$. By using Rankine-Hugoniot conditions
\begin{align*}
\begin{aligned}
&f(u_r)-f(u_l)=\sigma_{u_l}^1(s_0) (u_r-u_l),\\
&f(S_{u}^1(s))-f(u)= \sigma_u^1 (s) (S_{u}^1(s)-u),\\  
\end{aligned}
\end{align*} 
we rewrite it as
\begin{align*}
\begin{aligned}
&q(S_{u}^1(s),u_r) - \sigma_{u}^1(s) \eta(S_{u}^1(s)|u_r) \\
&\quad = \Big(-q(u_r,S_{u}^1(s_0)) + \sigma_{u}^1(s) \eta(u_r|S_{u}^1(s_0))\Big)+\Big(q(S_{u}^1(s),S_{u}^1(s_0)) - \sigma_{u}^1(s) \eta(S_{u}^1(s)|S_{u}^1(s_0))\Big)\\
&\qquad + \Big(\nabla\eta (u_r) -\nabla\eta (S_{u}^1(s_0))  \Big) \Big(f(u_l)-f(u) -\sigma_u^1 (s) (u_l-u) + (\sigma_{u_l}^1 (s_0)-\sigma_u^1 (s)) (u_r-u_l) \Big)\\
&\quad=:I_1 +I_2+I_3.
\end{aligned}
\end{align*} 
Since $u\rightarrow \sigma_{u}^1(s_0)$ and $u\rightarrow S_{u}^1(s_0)$ are $C^1$ and bounded in $B_{\eps_0}(u_l)$, we have
\begin{align*}
\begin{aligned}
& |q(S_{u_l}^1(s_0),S_{u}^1(s_0))|\le C |S_{u_l}^1(s_0)-S_{u}^1(s_0)|^2\le C |u_l- u|^2,\\
& \eta(S_{u_l}^1(s_0)|S_{u}^1(s_0))\le C |S_{u_l}^1(s_0)-S_{u}^1(s_0)|^2\le C |u_l- u|^2,\\
&|\nabla\eta (S_{u_l}^1(s_0)) -\nabla\eta (S_{u}^1(s_0))| \le C |S_{u_l}^1(s_0)-S_{u}^1(s_0)|\le C |u_l- u|,\\
& |\sigma_{u_l}^1 (s_0)-\sigma_u^1 (s_0)| \le C |u_l- u|.
\end{aligned}
\end{align*} 
This yields
\begin{align*}
\begin{aligned}
I_1 &=  -q(S_{u_l}^1(s_0),S_{u}^1(s_0)) + \sigma_{u}^1(s_0) \eta(S_{u_l}^1(s_0)|S_{u}^1(s_0))+(\sigma_{u}^1(s)-\sigma_{u}^1(s_0) )\eta(S_{u_l}^1(s_0)|S_{u}^1(s_0))\\
&\le C |u_l- u|^2 (1+ |\sigma_{u}^1(s)-\sigma_{u}^1(s_0)|),\\
\end{aligned}
\end{align*} 
and
\begin{align*}
\begin{aligned}
I_3 &= \Big(\nabla\eta (S_{u_l}^1(s_0)) -\nabla\eta (S_{u}^1(s_0))  \Big) \Big(f(u_l)-f(u) -\sigma_u^1 (s_0) (u_l-u)
-(\sigma_u^1 (s)-\sigma_u^1 (s_0)) (u_l-u)\\
&\quad+ (\sigma_{u_l}^1 (s_0)-\sigma_u^1 (s_0)) (u_r-u_l) -(\sigma_{u}^1 (s)-\sigma_u^1 (s_0)) (u_r-u_l) \Big),\\
&\le C |u_l- u| ( |u_l- u| + |\sigma_u^1 (s)-\sigma_u^1 (s_0)| |u_l- u| +|\sigma_u^1 (s)-\sigma_u^1 (s_0)| ).
\end{aligned}
\end{align*}
Thus using \eqref{delta-k}, if $|s-s_0|< \delta$, we have
\begin{align*}
\begin{aligned}
&q(S_{u}^1(s),u_r) - \sigma_{u}^1(s) \eta(S_{u}^1(s)|u_r)\\
&\quad\le -k |\sigma_{u}^1(s)-\sigma_{u}^1(s_0)|^2 \\
&\qquad + C|u-u_l|^2 (1+|\sigma_{u}^1(s)-\sigma_{u}^1(s_0)|) + C|u-u_l| |\sigma_{u}^1(s)-\sigma_{u}^1(s_0)|\\
&\quad\le C(|u-u_l|^2 +|u-u_l|^4) \le C|u-u_l|^2, 
\end{aligned}
\end{align*} 
where we have used Young's inequality and $u\in B_{\eps_0}(u_l)$.\\
If $|s-s_0|\ge \delta$, then we have
\begin{align*}
\begin{aligned}
&q(S_{u}^1(s),u_r) - \sigma_{u}^1(s) \eta(S_{u}^1(s)|u_r)\\
&\quad\le -k |\sigma_{u}^1(s)-\sigma_{u}^1(s_0)| \\
&\qquad + C|u-u_l|^2 (1+|\sigma_{u}^1(s)-\sigma_{u}^1(s_0)|) + C|u-u_l| |\sigma_{u}^1(s)-\sigma_{u}^1(s_0)|\\
&\quad\le C|u-u_l|^2, 
\end{aligned}
\end{align*} 
where we have used smallness of $\eps_0$ as $\eps_0\ll k$.\\
Therefore, we use \eqref{lem-5} and \eqref{sigma-con} to get
\begin{align*}
\begin{aligned}
&a \Big(q(S_{u}^1(s),u_r) - \sigma_{u}^1(s) \eta(S_{u}^1(s)|u_r) \Big) - q(u|u_l) +\sigma_{u}^1(s) \eta(u|u_l)\\
&\le aC|u-u_l|^2 - q(u|u_l) +\sigma_0 \eta(u|u_l)\\
&\le aC|u-u_l|^2 + \beta \eta(u|u_l).   
\end{aligned}
\end{align*}
Using Lemma \ref{lem-L2} and taking $a_*$ small enough such that $Ca_* \le \beta$ and still $R_{a_*}\subset B_{\eps_0}(u_l)$, we end up with \eqref{lem-rh} for all $a<a_*$ and $s\ge \frac{s_0}{2}$.\\
On the other hand, when $s< \frac{s_0}{2}$, since $\delta<\frac{s_0}{2}$ as in \eqref{delta-k}, we have $|s-s_0|>\delta$. Thus using smallness of $\eps_0$, we have
\begin{align*}
\begin{aligned}
&q(S_{u}^1(s),u_r) - \sigma_{u}^1(s) \eta(S_{u}^1(s)|u_r)\\
&\quad\le -k |\sigma_{u}^1(s)-\sigma_{u}^1(s_0)| \\
&\qquad + C|u-u_l|^2 (1+|\sigma_{u}^1(s)-\sigma_{u}^1(s_0)|) + C|u-u_l| |\sigma_{u}^1(s)-\sigma_{u}^1(s_0)|\\
&\quad\le -\frac{k}{2} |\sigma_{u}^1(s)-\sigma_{u}^1(s_0)|+ C|u-u_l|^2. 
\end{aligned}
\end{align*} 
Using \eqref{Liu-0} and smallness of $\eps_0$ again, we have that for all $s< \frac{s_0}{2}$,
\begin{align*}
\begin{aligned}
&q(S_{u}^1(s),u_r) - \sigma_{u}^1(s) \eta(S_{u}^1(s)|u_r)\\
&\quad\le -\frac{k}{2} |\sigma_{u}^1(\frac{s_0}{2})-\sigma_{u}^1(s_0)|+ C|u-u_l|^2\\
&\quad\le -\frac{k}{4} |\sigma_{u}^1(\frac{s_0}{2})-\sigma_{u}^1(s_0)|.
\end{aligned}
\end{align*}
Since $u\in R_a\subset B_{\eps_0}(u_l)$, we use \eqref{Liu-0} and \eqref{lem-5} to estimate
\begin{align}
\begin{aligned}\label{C_1}
- q(u|u_l) +\sigma_{u}^1(s) \eta(u|u_l) &\le - q(u|u_l) +\lambda_1(u) \eta(u|u_l)\\
&= - q(u|u_l) +\sigma_0 \eta(u|u_l) +(\lambda_1(u)-\sigma_0)\eta(u|u_l) \\
&\le (\lambda_1(u)-\sigma_0)\eta(u|u_l)\\
&\le a(\lambda_1(u)-\sigma_0)\eta(u|u_r)\\
&\le aC_1(\lambda_1(u)-\sigma_0).
\end{aligned}
\end{align} 
Therefore, for all $s<\frac{s_0}{2}$,
\begin{align*}
\begin{aligned}
&a \Big(q(S_{u}^1(s),u_r) - \sigma_{u}^1(s) \eta(S_{u}^1(s)|u_r) \Big) - q(u|u_l) +\sigma_{u}^1(s) \eta(u|u_l)\\
&\quad\le -a\Big(\frac{k}{4} |\sigma_{u}^1(\frac{s_0}{2})-\sigma_{u}^1(s_0)| - C(\lambda_1(u)-\sigma_0) \Big).
\end{aligned}
\end{align*} 
Hence we use \eqref{later-c} to conclude \eqref{lem-rh}. 
\end{proof}

\subsection{Main result}
We here revisit the main result in \cite{Vasseur-2013} showing that the hypotheses $(\mathcal{H})$ (resp. $(\mathcal{H}^*)$) implies {\it{a-SRES}}, thus $a$-contraction of $1$-shock (resp. $n$-shock) thanks to Theorem \ref{thm-cont-1}. 
\begin{theorem}\label{thm-ext}
Suppose that a system \eqref{main} satisfies $(\mathcal{H})$ and $(u_l,u_r, \sigma_{l,r})$ is a $1$-shock such that $u_r=S_{u_l}^1(s_0)$, $\sigma_{l,r}=\sigma_{u_l}^1(s_0)$ for some $s_0>0$. Then, there exists a small constant $0<a_*<1$ such that for all $0<a\le a_*$, $(u_l,u_r, \sigma_{l,r})$ is {\it{a-SRES}}, thus satisfies $a$-contraction.\\
As a dual result, if we suppose that a system \eqref{main} satisfies $(\mathcal{H}^*)$ and $(u_l,u_r, \sigma_{l,r})$ is a $n$-shock such that $u_l=S_{u_r}^n(s_0)$, $\sigma_{l,r}=\sigma_{u_r}^n(s_0)$ for some $s_0>0$. Then, there exists a large constant $a^*>1$ such that for all $a\ge a^*$, $(u_l,u_r, \sigma_{l,r})$ is {\it{a-SRES}}, thus satisfies $a$-contraction.
\end{theorem}
\begin{proof}
By Remark \ref{remark-dual}, we only prove that a given $1$-shock $(u_l,u_r, \sigma_{l,r})$ is {\it{a-SRES}}, which implies $a$-contraction by Theorem \ref{thm-cont-1}.\\
$\bullet$ {\bf Step A} (Verifying $(\mathcal{H}1)$) : 
We take $\sigma_0$, $\eps_0$ and $a_*$ as in Lemma \ref{lem-v2} such that for any $0<a<a_*$, \eqref{domain-r} is satisfied, which yields 
\[
\Sigma_a:=\{ u ~|~ \eta(u|u_l) = a\eta(u|u_r) \} \subset B_{\eps_0}(u_l).
\]
Therefore, using \eqref{lem-5}, we have that for any $u\in \Sigma_a$,
 \begin{align*}
\begin{aligned}
D_{sm}(u_{l,r};u) &= a q(u|u_r) - q(u|u_l)\\
&= a q(u|u_r) - q(u|u_l) +\sigma_0 (\eta(u|u_l) - a\eta(u|u_r)) \\
&= a \Big(q(u|u_r) -\sigma_0 \eta(u|u_r)\Big) -q(u|u_l)+\sigma_0 \eta(u|u_l)\\
&\le 0.
\end{aligned}
\end{align*}
$\bullet$ {\bf Step B} (Verifying $(\mathcal{H}2)$) : 
For any $1\le i\le n$, let $(u_-,u_+,\sigma_{\pm})$ be any $i$-th entropic discontinuity satisfying
\beq\label{1-case}
\eta(u_-|u_l) < a \eta(u_-|u_r) ~\mbox{and}~ \eta(u_+|u_l) > a \eta(u_+|u_r).
\eeq
Then by Lemma \ref{lem-L2}, the distance $|u_- - u_l|$ is estimated by weight $a$ as follows: 
\beq\label{a-small}
|u_- - u_l| \le C \sqrt{\eta(u_- | u_l)} < C \sqrt{a \eta(u_-|u_r)} \le C\sqrt{a}.
\eeq
For such shocks $(u_-,u_+,\sigma_{\pm})$, we consider
\begin{align}
\begin{aligned}\label{D-RH}
D_{RH}(u_{l,r},u_{\pm}) &= a \underbrace{\Big(q(u_+|u_r) - \sigma_{\pm}  \eta(u_+|u_r) \Big)}_{\mathcal{R}}+ \underbrace{\Big(- q(u_-|u_l) +\sigma_{\pm} \eta(u_-|u_l)\Big)}_{\mathcal{L}}.
\end{aligned}
\end{align}
{\bf Case of the $1$-shock $(u_-,u_+,\sigma_{\pm})$} : 
If $(u_-,u_+,\sigma_{\pm})$ is a $1$-shock, we take $\eps_0>0$ and $a_*>0$ as in Lemma \ref{lem-v2} such that \eqref{lem-rh} holds for all $u_-\in R_a\subset B_{\eps_0}(u_l)$, which yields
\[
D_{RH}(u_{l,r},u_{\pm})\le 0.
\]
where $u_+=S_{u_-}^1(s)$ and $\sigma_{\pm}=\sigma_{u_-}^1(s)$.\\
{\bf Case of other families $i\ge 2$} :
If $(u_-,u_+,\sigma_{\pm})$ is any $i$-th entropic discontinuity for $i\ge 2$, we need to compute $\mathcal{L}$ and $\mathcal{R}$ in \eqref{D-RH} further as follows.\\
To estimate $\mathcal{L}$, we use Taylor expansion with respect to $u_-$ at $u_l$, to get
\begin{align}
\begin{aligned}\label{same-1}
&-q(u_-|u_l) +  \lambda_1(u_-)  \eta(u_-|u_l) \\
&\quad=  (u_- -u_l)^{T} \nabla^2 \eta(u_l) (\lambda_1(u_l) I -\nabla f(u_l)) (u_- -u_l) + \mathcal{O}(|u_--u_l|^3).
\end{aligned}
\end{align}
Since the entropy $\eta$ is strictly convex, $\nabla^2\eta(u_l)$ is symmetric, strictly positive and $ \nabla^2 \eta(u_l)\nabla f(u_l))$ is symmetric. Thus those matrices are diagonalizable in the same basis, which gives
\[ 
\nabla^2 \eta(u_l)\nabla f(u_l)) \ge \lambda_1(u_l) \nabla^2 \eta(u_l).
\]
Thus together with \eqref{a-small}, we have
\[
-q(u_-|u_l) +  \lambda_1(u_-)  \eta(u_-|u_l) \le Ca^{3/2},
\]
which yields
\begin{align}
\begin{aligned}\label{L-est}
\mathcal{L} &=-q(u_-|u_l) +  \lambda_1(u_-)  \eta(u_-|u_l) + (\sigma_{\pm} -  \lambda_1(u_-))  \eta(u_-|u_l)  \\
&\le (\sigma_{\pm} -  \lambda_1(u_-))  \eta(u_-|u_l) +Ca^{3/2}.
\end{aligned}
\end{align}
On the other hand, we use the identity \eqref{metric} and Rankine-Hugoniot condition to have
\begin{align*}
\begin{aligned}
\mathcal{R} &= \Big( q(u_+|u_-) - \sigma_{\pm}  \eta(u_+|u_-) \Big)+ \Big( q(u_-|u_r) - \sigma_{\pm}  \eta(u_-|u_r)\Big)\\
&\quad -  (d\eta(u_-)-d\eta(u_r)) \cdot \Big(f(u_-)-f(u_+) -\sigma_{\pm} (u_--u_+) \Big)\\
&= \Big( q(u_+|u_-) - \sigma_{\pm}  \eta(u_+|u_-) \Big)+ \Big( q(u_-|u_r) - \sigma_{\pm}  \eta(u_-|u_r)\Big)\\
&=: I_1 +I_2.
\end{aligned}
\end{align*}
Since $(u_-,u_+,\sigma_{\pm})$ is entropic discontinuity, it follows from \eqref{f-q} that $I_1\le 0$.\\
Using \eqref{metric} again, we have
\begin{align*}
\begin{aligned}
I_2 &= \Big( q(u_-|u_l) - \sigma_{\pm}  \eta(u_-|u_l) \Big)+ \Big( q(u_l|u_r) - \sigma_{\pm}  \eta(u_l|u_r)\Big)\\
&\quad -  (d\eta(u_l)-d\eta(u_r)) \cdot \Big(f(u_l)-f(u_-) -\sigma_{\pm} (u_l -u_-) \Big)\\
&=: I_{21} +I_{22} +I_{23}.  
\end{aligned}
\end{align*}
We use \eqref{a-small} to get
\begin{align*}
\begin{aligned}
&I_{21} \le C|u_- -u_l|^2\le Ca,\\
&I_{23} \le C|u_- -u_l|\le C\sqrt{a}.
\end{aligned}
\end{align*}
To estimate $I_{22}$, we apply \eqref{vasseur-1} with $u=u_l$ and $s=0$ to the first shock $(u_l,u_r,\sigma_{l,r})$ so that
\[
q(u_l|u_r) - \lambda_1(u_l)  \eta(u_l|u_r)= \int^0_{s_0} \frac{d}{dt}\sigma^1_{u_1}(t)\Big( \eta(u_l|S_{u_l}^1(t))-\eta(u_l|S_{u_l}^1(s_0)) \Big) dt. 
\]
By the conditions $\frac{d}{dt}\sigma^1_{u_1}(t)>0$ and $\frac{d}{dt}\eta(u|S_{u_l}^1(t))<0$ in \eqref{Liu-0}, we have
\beq\label{good}
q(u_l,u_r)   -  \lambda_1(u_l)  \eta(u_l|u_r)=:-C_{l,r} <0.
\eeq
This and the smoothness of $\lambda_1$ yields
\begin{align*}
\begin{aligned}
I_{22}&= q(u_l|u_r) -  \lambda_1(u_l)  \eta(u_l|u_r)  +(\lambda_1(u_l)-\lambda_1(u_-)) \eta(u_l|u_r) +(\lambda_1(u_-)-\sigma_{\pm}) \eta(u_l|u_r)  \\
&= -C_{l,r} + C|u_- -u_l| + (\lambda_1(u_-)-\sigma_{\pm}) \eta(u_l|u_r)\\
&\le -C_{l,r} + C\sqrt{a} + (\lambda_1(u_-)-\sigma_{\pm}) \eta(u_l|u_r).
\end{aligned}
\end{align*}
Therefore, by smallness of $a$, we have
\beq\label{R}
\mathcal{R} \le -\frac{1}{2}C_{l,r} + (\lambda_1(u_-)-\sigma_{\pm}) \eta(u_l|u_r).
\eeq
Therefore we combine \eqref{L-est} and \eqref{R} to get
\begin{align}
\begin{aligned}\label{rh-1}
D_{RH}(u_{l,r},u_{\pm})&\le -\frac{a}{2}C_{l,r} + a(\lambda_1(u_-)-\sigma_{\pm})\eta(u_l|u_r)\\
&\quad +(\sigma_{\pm} -  \lambda_1(u_-))  \eta(u_-|u_l) + Ca^{3/2}. 
\end{aligned}
\end{align}
We use \eqref{a-small} to estimate
\begin{align}
\begin{aligned}\label{a-uu}
\eta(u_-|u_l) &< a \eta(u_-|u_r) \\
& = a \eta(u_l|u_r) + a (\eta(u_-) -\eta(u_l) -d\eta(u_r)\cdot(u_--u_l))\\
& =a \eta(u_l|u_r) + aC|u_- - u_l|\\
& <a \eta(u_l|u_r) + Ca^{3/2}.
\end{aligned}
\end{align}
Since $\lambda_1(u_-)\le \sigma_{\pm}$ by \eqref{Lax-0} for $i\ge 2$, we combine  \eqref{rh-1} and \eqref{a-uu} to estimate
\begin{align*}
\begin{aligned}
D_{RH}(u_{l,r},u_{\pm})&< -\frac{a}{2}C_{l,r} + a(\lambda_1(u_-)-\sigma_{\pm})\eta(u_l|u_r)\\
&\quad +(\sigma_{\pm} -  \lambda_1(u_-))\Big( a\eta(u_l|u_r)+ Ca^{3/2}\Big) + Ca^{3/2}\\
&=-\frac{a}{2}C_{l,r} + Ca^{3/2}\\
&\le 0.
\end{aligned}
\end{align*}
$\bullet$ {\bf Step C} (Verifying $(\mathcal{H}2^*)$) : 
It remains to consider the case that for $1\le i\le n$, $(u_-,u_+,\sigma_{\pm})$ is the $i$-th the entropic discontinuity satisfying 
\beq\label{con-1}
\eta(u_-|u_l) > a \eta(u_-|u_r) ~\mbox{and}~ \eta(u_+|u_l) < a \eta(u_+|u_r).
\eeq
Then by Lemma \ref{lem-L2}, the distance $|u_- - u_l|$ is estimated by weight $a$ as follows.
\beq\label{a-s}
|u_+ - u_l| \le C \sqrt{\eta(u_+ | u_l)} \le C \sqrt{a \eta(u_+|u_r)} = C\sqrt{a}.
\eeq
We apply the identity \eqref{metric} to $\mathcal{R}$ in \eqref{D-RH} to get
\begin{align*}
\begin{aligned}
\mathcal{R} &= \Big( q(u_+|u_l) - \sigma_{\pm}  \eta(u_+|u_l) \Big)+ \Big( q(u_l|u_r) - \sigma_{\pm}  \eta(u_l|u_r)\Big)\\
&\quad -  (d\eta(u_l)-d\eta(u_r)) \cdot \Big(f(u_l)-f(u_+) -\sigma_{\pm} (u_l -u_+) \Big)\\
&=: I_1 +I_2+I_3.
\end{aligned}
\end{align*}
By \eqref{a-s}, we have
\begin{align*}
\begin{aligned}
&I_{1} \le C|u_+ -u_l|^2\le C a,\\
&I_{3} \le C|u_+ -u_l|\le C\sqrt{a}.
\end{aligned}
\end{align*}
We use \eqref{good}, \eqref{a-s} the smoothness of $\lambda_1$ to get
\begin{align*}
\begin{aligned}
I_{2}&= q(u_l|u_r) -  \lambda_1(u_l)  \eta(u_l|u_r)  +(\lambda_1(u_l)-\lambda_1(u_+)) \eta(u_l|u_r) +(\lambda_1(u_+)-\sigma_{\pm}) \eta(u_l|u_r)  \\
&= -C_{l,r} +C|u_+ -u_l| + (\lambda_1(u_+)-\sigma_{\pm}) \eta(u_l|u_r)\\
&\le -C_{l,r} + C\sqrt{a} + (\lambda_1(u_+)-\sigma_{\pm}) \eta(u_l|u_r).
\end{aligned}
\end{align*}
Thus by smallness of $a$, we have
\beq\label{R-1}
\mathcal{R}\le -\frac{C_{l,r}}{2} + (\lambda_1(u_+)-\sigma_{\pm}) \eta(u_l|u_r). 
\eeq
On the other hand, applying the identity \eqref{metric} to $\mathcal{L}$ in \eqref{D-RH}, we have
\begin{align*}
\begin{aligned}
\mathcal{L} &= \Big( -q(u_-|u_+) + \sigma_{\pm}  \eta(u_-|u_+) \Big) +\Big( -q(u_+|u_l) + \sigma_{\pm}  \eta(u_+|u_l)\Big)\\
&\quad +  (d\eta(u_+)-d\eta(u_-)) \cdot \Big(f(u_+)-f(u_-) -\sigma_{\pm} (u_+ -u_-) \Big)\\
&=: J_1 + J_2 + J_3.
\end{aligned}
\end{align*}
By \eqref{f-q} for entropic discontinuity $(u_-,u_+,\sigma_{\pm})$, we have 
\[J_1\le 0,\quad J_3=0,
\]
thus
\[
\mathcal{L} \le - q(u_+|u_l) + \sigma_{\pm}  \eta(u_+|u_l).
\]
We now combine this and \eqref{R-1} to have
\begin{align*}
\begin{aligned}
D_{RH}(u_{l,r},u_{\pm})\le -\frac{a}{2}C_{l,r} + a(\lambda_1(u_+)-\sigma_{\pm})\eta(u_l|u_r)- q(u_+|u_l) + \sigma_{\pm}  \eta(u_+|u_l).
\end{aligned}
\end{align*}
Since \eqref{con-1} and \eqref{a-s} yield
\begin{align}
\begin{aligned}\label{a-11}
\eta(u_+|u_l) &< a \eta(u_+|u_r) \\
& = a \eta(u_l|u_r) + a (\eta(u_+) -\eta(u_l) -d\eta(u_r)\cdot(u_+ -u_l))\\
& =a \eta(u_l|u_r) + aC|u_+ - u_l|\\
& \le a \eta(u_l|u_r) + Ca^{3/2},
\end{aligned}
\end{align}
we use the fact that $\lambda_1(u_+)\le \sigma_{\pm}$ by \eqref{Lax-0}, to have
\begin{align*}
\begin{aligned}
D_{RH}(u_{l,r},u_{\pm})&< -\frac{a}{2}C_{l,r} + (\lambda_1(u_+)-\sigma_{\pm})\eta(u_+|u_l)+ Ca^{3/2} 
- q(u_+|u_l)+ \sigma_{\pm}  \eta(u_+|u_l)\\
&=-\frac{a}{2}C_{l,r} - q(u_+|u_l) + \lambda_1(u_+)\eta(u_+|u_l)+ Ca^{3/2} 
\end{aligned}
\end{align*}
Moreover, using the same computation as \eqref{same-1}, we get
\begin{align}
\begin{aligned}\label{same-1}
&-q(u_+|u_l) +  \lambda_1(u_+)  \eta(u_+|u_l) \\
&\quad =  (u_+ -u_l)^{T} \nabla^2 \eta(u_l) (\lambda_1(u_+) I -\nabla f(u_l)) (u_+ -u_l) + C|u_+-u_l|^3\\
&\quad \le (\lambda_1(u_+) -\lambda_1(u_l))  (u_+ -u_l)^{T} \nabla^2 \eta(u_l) (u_+ -u_l) + C|u_+ -u_l|^3\\
&\quad \le C|u_+ -u_l|^3,
\end{aligned}
\end{align}
which yields
\begin{align*}
\begin{aligned}
D_{RH}(u_{l,r},u_{\pm})&\le -\frac{a}{2}C_{l,r} +Ca^{3/2} \\
& \le 0.
\end{aligned}
\end{align*}
\end{proof}

\section{Criteria preventing  $a$-contractions for intermediate entropic shocks}\label{no-cont-shock}
In this section, as an application of Theorem \ref{thm-cont-2}, we construct some sufficient conditions, which provide no $a$-contraction for the intermediate entropic shocks. For that, we consider the following hypotheses $(\mathcal{A}_1)$.
\bigskip

\begin{itemize}
\item $(\mathcal{A}_1)$ : We suppose that for some $1<i<n$, the $i$-th intermediate characteristic field $(\lambda_i, r_i)$ is genuinely nonlinear on $\mathcal{U}$. For a given $u \in \mathcal{V}$, suppose that there is a $C^1$ $i$-th Hugoniot curve $S_{u}^i(s) \in \mathcal{U}$ defined on an interval $[0,s_u)$ (possibly $s_u=\infty$), such that 
$S_{u}^i(0)=u$ and the Rankine-Hugoniot condition:
\[f(S_{u}^i(s)) -f(u) = \sigma_{u}(s) (S_{u}^i(s) - u),\]
where $\sigma_{u}(s)$ is a $C^1$ velocity function satisfying the Liu entropy condition:
\beq\label{Liu-1}
\sigma_{u}^{\prime}(s) < 0,
\eeq
and $\sigma_{u}(0)=\lambda_i(u)$.\\
Moreover, we suppose that the system \eqref{main} satisfies the extended Lax condition:
\beq\label{Lax}
\lambda_{l-1} (u_-) \le \sigma \le \lambda_{l+1} (u_+),\quad \mbox{for any $l$-th shock $(u_-,u_+,\sigma)$}.
\eeq
\end{itemize}

\subsection{Conditions via neighboring genuinely nonlinear fields}
The following theorem is on the case where there is a genuinely nonlinear field except for the intermediate entropic shock. For this case, we construct some sufficient conditions, which provide no $a$-contraction for intermediate entropic shock. 
\begin{theorem}\label{thm-shock}
Consider the system \eqref{main} satisfying $(\mathcal{A}_1)$. Let $(u_l,u_r, \sigma_{l,r})$ be a given $i$-th the entropic shock such that $u_r=S_{u_l}^i(s_0)$, $\sigma_{u_l}(s_0)=\sigma_{l,r}$ for some $s_0>0$ and the conditions \eqref{Liu-1} and \eqref{Lax} are satisfied. Then the following statements holds.\\
\begin{itemize}
\item (1) For $0<a<1$, we assume that there is a $C^1$ $j$-th rarefaction curve $R_{u_l}^j (s)$ with $j<i$ such that $R_{u_l}^j(0)=u_l$ and the backward curve $R_{u_l}^{j,-} (s)$ of $R_{u_l}^j (s)$, i.e., $\lambda_j (R_{u_l}^{j,-} (s)) < \lambda_j (u_l) $, intersects with the $(n-1)$-dimensional surface $\Sigma_a$.
Then, the entropic shock $(u_l,u_r, \sigma_{l,r})$ does not satisfy $a$-contraction.\\

\item (2) For $a>1$, we assume that there is a $C^1$ $k$-th rarefaction curve $R_{u_r}^k (s)$ with $k>i$ such that $R_{u_r}^k(0)=u_r$ and the forward curve $R_{u_r}^{k,+} (s)$ of $R_{u_r}^k (s)$, i.e., $\lambda_k (R_{u_r}^{k,+} (s)) > \lambda_k (u_r) $, intersects with the $(n-1)$-dimensional surface $\Sigma_a$.
Then, the entropic shock $(u_l,u_r, \sigma_{l,r})$ does not satisfy $a$-contraction.\\

\item  (3) For $a=1$, we assume that one of the assumptions of (1) and (2) is satisfied. Then, the entropic shock $(u_l,u_r, \sigma_{l,r})$ does not satisfy $a$-contraction.\\
\end{itemize}
\end{theorem}

\begin{remark}
The reason of splitting $a>0$ into the three ranges as above is motivated from the geometric observation. More precisely, the $(n-1)$-dimensional surface $\Sigma_1$ for $a=1$ becomes hyperplane as
\[
(d\eta(u_r) - d\eta(u_l))\cdot u = d\eta(u_r)\cdot u_ r -\eta(u_r) -d\eta(u_l)\cdot u_ l +\eta(u_l),
\]
which separates $u_l$ and $u_r$ in the phase space, whereas the surface $\Sigma_a$ for $0<a<1$ (resp. $a>1$) is strictly convex that belongs to the region including $u_l$ (resp. $u_r$). Thus, for the case of $0<a<1$ (resp. $a>1$), the rarefaction curve $R_{u_l}^{j,-} (s)$ (resp. $R_{u_r}^{k,+} (s)$) issued from $u_l$ (resp. $u_r$) is likely to intersect with $\Sigma_a$. On the other hand, since $\Sigma_a$ shrinks to $u_l$ (resp. $u_r$) as $a\rightarrow 0$ (resp. $a\rightarrow\infty$), the rarefaction curve $R_{u_l}^{j,-} (s)$ (resp. $R_{u_r}^{k,+} (s)$) hardly intersects with $\Sigma_a$ for $a\gg 1$ (resp. $a\ll 1$). In fact, the proof of Theorem \ref{thm-contact-2} does not depend on the strength of weight $a$. 
\end{remark}

{\bf Proof of Theorem \ref{thm-contact-2}} 
{\bf(1) Case of $0<a<1$ :} Let $\bar{u}$ be the first intersection point of the $j$-th the backward rarefaction curve 
$R_{u_l}^{j,-} (s)$ and the surface $\Sigma_a$. Let us put $\bar{u}=R_{u_l}^{j,-} (\bar{s})$ for some $\bar{s}>0$.
We show that the shock $(u_l,u_r, \sigma_{l,r})$ does not satisfies the first condition $(\mathcal{H} 1)$ of {\it{a-RES}} as
\begin{align*}
\begin{aligned}
D_{sm}(u_{l,r};\bar{u}) = a q(\bar{u},u_r) -q(\bar{u},u_l) >0.
\end{aligned} 
\end{align*}
For this end, we consider a differentiable function $f : \bbr_+ \rightarrow \bbr$ defined by
\begin{align}
\begin{aligned}\label{function}
F(s) &= a \Big(q(R_{u_l}^{j,-} (s), u_r) -\lambda_j(R_{u_l}^{j,-} (s)) \eta(R_{u_l}^{j,-} (s) | u_r) \Big)\\
&\qquad -q(R_{u_l}^{j,-} (s),u_l) 
+ \lambda_j(R_{u_l}^{j,-} (s)) \eta(R_{u_l}^{j,-} (s) | u_l).
\end{aligned} 
\end{align}
Note that we have
\[
F(\bar{s}) = D_{sm}(u_{l,r};\bar{u}),
\]
because $a\eta(\bar{u} | u_r)=\eta(\bar{u} | u_l)$.\\
Since the $i$-th shock $(u_l, S_{u_l}^i(s),\sigma_{u_l}^i(s))$ satisfies \eqref{Liu-1} and \eqref{Lax}, by the inequality
\eqref{vasseur} and the fact $\lambda_j(u_l)<\lambda_i(u_l)$, we have
\begin{align}
\begin{aligned}\label{since-1}
& q(u_l, u_r) -\sigma_{l,r}\eta(u_l | u_r) =- \int_0^{s_0} \frac{d}{dt}\sigma^{i}_{u_l}(t)\eta(u_l|S_{u_l}(t))dt > 0,\\
& \sigma_{l,r} \ge \lambda_{i-1} (u_l) \ge \lambda_j (u_l),
\end{aligned} 
\end{align}
which implies the positivity of $f(0)$ by
\begin{align}
\begin{aligned}\label{f-zero}
F(0) & = a \Big(q(u_l, u_r) -\lambda_j(u_l) \eta(u_l | u_r) \Big)\\
& = a \Big(q(u_l, u_r) -\sigma_{l,r} \eta(u_l | u_r) \Big)  +a(\sigma_{l,r}- \lambda_j(u_l)) \eta(u_l | u_r) \\
&> 0.
\end{aligned} 
\end{align}
Let us show 
\beq\label{f-prime}
F^{\prime} (s) >0,\quad 0< s < \bar{s}.
\eeq
First of all, since $\frac{d}{ds}R_{u_l}^{j,-} (s) = r_j (R_{u_l}^{j,-} (s))$, we have
\[
\nabla f (R_{u_l}^{j,-} (s)) \frac{d}{ds}R_{u_l}^{j,-} (s) = \lambda_j(R_{u_l}^{j,-} (s))\frac{d}{ds}R_{u_l}^{j,-} (s),
\]
which yields
\begin{align*}
\begin{aligned}
F^{\prime}(s) & = a \Big[ (d\eta (R_{u_l}^{j,-} (s))-d\eta(u_r))\Big(\nabla f (R_{u_l}^{j,-} (s)) -\lambda_j(R_{u_l}^{j,-} (s)) I \Big) \frac{d}{ds}R_{u_l}^{j,-} (s) \Big]\\
&\quad -(d\eta (R_{u_l}^{j,-} (s))-d\eta(u_l))\Big(\nabla f (R_{u_l}^{j,-} (s)) -\lambda_j(R_{u_l}^{j,-} (s)) I \Big) \frac{d}{ds}R_{u_l}^{j,-} (s)\\
&\quad +\frac{d}{ds}\lambda_j(R_{u_l}^{j,-} (s)) \Big( \eta (R_{u_l}^{j,-} (s) | u_l) - a\eta (R_{u_l}^{j,-} (s) | u_r) \Big)\\
&=\frac{d}{ds}\lambda_j(R_{u_l}^{j,-} (s)) \Big( \eta (R_{u_l}^{j,-} (s) | u_l) - a\eta (R_{u_l}^{j,-} (s) | u_r) \Big).
\end{aligned} 
\end{align*}
Here, since $\frac{d}{ds}\lambda_j(R_{u_l}^{j,-} (s))$ is continuous and $\lambda_j(R_{u_l}^{j,-} (s)) < \lambda_j(u_l)$ besides
\[
\frac{d}{ds}\lambda_j(R_{u_l}^{j,-} (s)) = \nabla \lambda_j(R_{u_l}^{j,-} (s)) \cdot r_j (R_{u_l}^{j,-} (s))\neq 0,
\]
we have
\[
\frac{d}{ds}\lambda_j(R_{u_l}^{j,-} (s)) <0,\quad s>0.
\]
Since $R_{u_l}^{j,-} (\bar{s})=\bar{u}$ is the first intersection point of the continuous curve $R_{u_l}^{j,-} (s)$ and the surface $\Sigma_a$, we have
\[
\eta (R_{u_l}^{j,-} (s) | u_l)  - a\eta (R_{u_l}^{j,-} (s) | u_r) <0,\quad 0<s<\bar{s}.
\]
Thus we have shown \eqref{f-prime}, which implies together with \eqref{f-zero} that
\[
F(\bar{s}) = D_{sm}(u_{l,r};\bar{u})>0.
\]
Hence the shock $(u_l,u_r, \sigma_{l,r})$ is not {\it{a-RES}}, which provides the conclusion by Theorem \ref{thm-cont-2}.\\
{\bf(2) Case of $a>1$ :} We follow the same argument as above by considering $R_{u_r}^{k,+}$ instead of $R_{u_l}^{j,-}$. Indeed, we use the same computations to have
\begin{align*}
\begin{aligned}
f(0) & = -q(u_r, u_l) + \lambda_k(u_r) \eta(u_r | u_l) \\
& = - \Big(q(u_r, u_l) -\sigma_{l,r} \eta(u_r | u_l) \Big)  +(\lambda_k(u_r)-\sigma_{l,r}) \eta(u_r | u_l) \\
&>0,
\end{aligned} 
\end{align*}
and 
\begin{align*}
\begin{aligned}
f^{\prime}(s) & =\frac{d}{ds}\lambda_j(R_{u_r}^{k,+} (s)) \Big( \eta (R_{u_r}^{j,+} (s) | u_l) - a\eta (R_{u_l}^{j,+} (s) | u_r) \Big)\\
&>0,\quad  0\le s\le \tilde{s},
\end{aligned} 
\end{align*}
where $\tilde{u}:=R_{u_r}^{k,+} (\tilde{s})$ is the first intersection point of the forward curve $R_{u_r}^{k,+} (s)$ and the surface $\Sigma_a$. Thus we have the same conclusion as
\[
f(\tilde{s}) = D_{sm}(u_{l,r};\tilde{u}) > 0.
\]
{\bf(3) Case of $a=1$ :} Since the proof of two cases above does not depend on the strength of $a$, if one of the assumptions of (1) and (2) is satisfied, we end up with no $a$-contraction of the entropic shock $(u_l,u_r, \sigma_{l,r})$.

\subsection{Application to magnetohydrodynamics}
As an application of Theorem \ref{thm-shock}, we here show that there is no contraction property of certain intermediate shocks for two-dimensional (planar) isentropic MHD in Lagrangian coordinates: 
\begin{align}
\begin{aligned}\label{MHD}
\left\{ \begin{array}{ll}
       \partial_t v - \partial_x u =0\\
        \partial_t (vB) - \beta \partial_x w = 0\\
         \partial_t u + \partial_x (p+\frac{1}{2}B^2) = 0\\
          \partial_t w - \beta \partial_x B = 0, \end{array} \right.
\end{aligned}
\end{align}
where $v$ denotes specific volume, and two-dimensional fluid velocity $(u, w)$ and magnetic field $(\beta,B)$ only depend on a single direction $e_1$ measured by $x$. This behavior of two-dimensional vector fields with spatially one-dimensional dependence is achieved when the initial condition is so. Thus the divergence-free condition of magnetic field of full MHD reduces that $\beta$ is constant (See for example \cite{B-L-Zumbrun} for study on \eqref{MHD}).  As a perfect fluid, the pressure $p$ is assumed to satisfies
\beq\label{pressure-MHD}
p(v)=v^{-\gamma},\quad \gamma >1.
\eeq
This system has an entropy $\eta$ as
\[
\eta(U)=\int_{v}^{\infty} p(s) ds + \frac{1}{2}(u^2 + w^2) + \frac{q^2}{2v}
\]
in terms of the conservative variables $U:=(v,q,u,w)$ where $q:=vB$.\\
For simplicity of computation, we use non-conservative variable $W:=(v,B,u,w)$ and rewrite \eqref{MHD} as a quasilinear form:
\[
 \partial_t W + A\partial_x W =0,
\]
where the $4\times 4$ matrix $A$ is given by
\[
A:= \left( \begin{matrix} 0&0&-1&0\\ 0&0&\frac{B}{v}&-\frac{\beta}{v}\\-c^2& B& 0&0\\ 0&-\beta &0&0
 \end{matrix} \right).
\]
where $c:=\sqrt{-p^{\prime}(v)}$ denotes the sound speed.\\
Since the eigenvalues of $A$ solves the characteristic polynomial
\[
\lambda^4 -\Big(\frac{B^2 + \beta^2}{v} + c^2 \Big) \lambda^2 + \frac{\beta^2}{v}c^2 =0,
\]
we have four eigenvalues
\[
\lambda_1=-\sqrt{\alpha_+},\quad \lambda_2=-\sqrt{\alpha_-},\quad \lambda_3=\sqrt{\alpha_-},\quad \lambda_4=\sqrt{\alpha_+},    
\]
where 
\[
\alpha_{\pm}:=\frac{1}{2} \Big[ \frac{B^2+\beta^2}{v} +c^2 \pm\sqrt{ \Big( \frac{B^2+\beta^2}{v}+c^2 \Big)^2-4\beta^2\frac{c^2}{v}} \Big].
\]
By a straightforward computation, we have the corresponding eigenvectors
\begin{align*}
\begin{aligned}
&r_1=\Big(1,-\frac{(\alpha_{+}-c^2)}{B}, \sqrt{\alpha_{+}},  -\frac{\beta(\alpha_{+}-c^2)}{B\sqrt{\alpha_+}} \Big)^T,\\
&r_2=\Big(1,-\frac{(\alpha_{-}-c^2)}{B}, \sqrt{\alpha_{-}},  -\frac{\beta(\alpha_{-}-c^2)}{B\sqrt{\alpha_-}} \Big)^T,\\
&r_3=\Big(-1,\frac{(\alpha_{-}-c^2)}{B}, \sqrt{\alpha_{-}},  -\frac{\beta(\alpha_{-}-c^2)}{B\sqrt{\alpha_-}} \Big)^T,\\
&r_4=\Big(-1,\frac{(\alpha_{+}-c^2)}{B}, \sqrt{\alpha_{+}},  -\frac{\beta(\alpha_{+}-c^2)}{B\sqrt{\alpha_+}} \Big)^T.
\end{aligned}
\end{align*}
Here, we restrict our study to the case of $B\neq0$. Using the strict convexity of pressure $p(v)$, $d\lambda_i\cdot r_i > 0$ for all $1\le i\le 4$, which means that all characteristic fields are genuinely non-linear. Indeed since
\[
\alpha_{\pm} -c^2=\frac{1}{2} \Big[ \frac{B^2+\beta^2}{v} -c^2 \pm\sqrt{ \Big( \frac{B^2+\beta^2}{v}-c^2 \Big)^2+4\frac{B^2c^2}{v}} \Big]
\]
yields
\beq\label{alpha-c}
\alpha_{-} < c^2 < \alpha_+,
\eeq
thus we have
\begin{align*}
\begin{aligned}
d\lambda_1\cdot r_1 =\frac{1}{2\sqrt{\alpha_+}}\Big[-\partial_v\alpha_+ +\partial_B\alpha_+\frac{\alpha_{+}-c^2}{B} \Big],
\end{aligned}
\end{align*}
\begin{align*}
\begin{aligned}
-\partial_v\alpha_+ &=\frac{1}{2}\Big[\frac{B^2+\beta^2}{v^2}+p^{\prime\prime}+\frac{(\frac{B^2+\beta^2}{v}-c^2)(\frac{B^2+\beta^2}{v^2}-p^{\prime\prime})+\frac{2B^2c^2}{v^2}+\frac{2B^2p^{\prime\prime}}{v}}{\sqrt{ \Big( \frac{B^2+\beta^2}{v}-c^2 \Big)^2+4\frac{B^2c^2}{v}}} \Big]\\
&>\frac{1}{2}\frac{\frac{2B^2c^2}{v^2}+\frac{2B^2p^{\prime\prime}}{v}}{\sqrt{ \Big( \frac{B^2+\beta^2}{v}-c^2 \Big)^2+4\frac{B^2c^2}{v}}}>0,
\end{aligned}
\end{align*}
and 
\begin{align*}
\begin{aligned}
\partial_B\alpha_+\frac{\alpha_{+}-c^2}{B} &= \Big[\frac{1}{v}+\frac{(\frac{B^2+\beta^2}{v}-c^2)\frac{1}{v}+\frac{2c^2}{v}}{\sqrt{ \Big( \frac{B^2+\beta^2}{v}-c^2 \Big)^2+4\frac{B^2c^2}{v}}}  \Big](\alpha_{+}-c^2)\\
&>\frac{\frac{2c^2}{v}}{\sqrt{ \Big( \frac{B^2+\beta^2}{v}-c^2 \Big)^2+4\frac{B^2c^2}{v}}}(\alpha_{+}-c^2)>0.
\end{aligned}
\end{align*}
Similarly we have $d\lambda_i\cdot r_i > 0$ for $i=2,3,4$.\\
Let $(U_l,U_r,\sigma_2)$ be the 2-shock wave satisfying the Rankine-Hugoniot condition:
\begin{align}
\begin{aligned}\label{MHD-RH}
-[u]&=\sigma_2[v],\\
-\beta [w]&= \sigma_2 [q],\\
[p]+\Big[\frac{q^2}{2v^2}\Big]&=\sigma_2 [u],\\
-\beta \Big[\frac{q}{v}\Big]&= \sigma_2 [w],
\end{aligned}
\end{align}
where $[f]:=f_r-f_l$.\\
By Lax condition, $d\lambda_2\cdot r_2 >0$ implies that $-r_2(U_l)$ is a tangent vector at $U_l$ of the 2-shock curve $S_{U_l}^2$ issuing from $U_l$. Thus since $dv\cdot(-r_2)<0$ and $du\cdot(-r_2)<0$, we have
\beq\label{2-shock}
[v]<0\quad \mbox{and}\quad [u]<0. 
\eeq
Since
\begin{align*}
\begin{aligned}
dB\cdot (-r_2)=\frac{(\alpha_{-}-c^2)}{B}=\left\{ \begin{array}{ll}
       <0\quad\mbox{if} ~B>0\\
        >0\quad\mbox{if} ~B<0, \end{array} \right.
\end{aligned}
\end{align*}
we have that $[B]<0$ for $B_l>0$, and $[B]>0$ for $B_l<0$. In particular, we here consider the case where the 2-shock wave satisfies
\beq\label{2-B}
\mbox{either}\quad B_l>B_r>0\quad \mbox{or}\quad B_l<B_r<0. 
\eeq
Similarly for a given 3-shock wave $(\tilde{U}_l, \tilde{U}_r, \sigma_3)$, we have $[\tilde{v}]>0$, $[\tilde{u}]<0$, and
$[\tilde{B}]>0$ for $\tilde{B}_l>0$, and $[\tilde{B}]<0$ for $\tilde{B}_l<0$. Also, we consider the case where the 3-shock wave 
satisfies
\beq\label{3-B}
\mbox{either}\quad \tilde{B}_r>\tilde{B}_l>0\quad \mbox{or}\quad \tilde{B}_r<\tilde{B}_l<0. 
\eeq

We are now ready to show that for any $a>0$, there is no $a$-contraction of such intermediate shocks as follows.\\\\
\begin{theorem}\label{thm-MHD}
Let $(U_l, U_r, \sigma_2)$ be a given 2-shock wave of the system \eqref{MHD}-\eqref{pressure-MHD}
satisfying \eqref{2-B}. Then there is no weight $a>0$ such that $(u_l,u_r)$ satisfies $a$-contraction. Likewise, this result holds for a given 3-shock wave $(\tilde{U}_l, \tilde{U}_r, \sigma_3)$ satisfying \eqref{3-B}.
\end{theorem}
\begin{proof}
First of all, we show that for any $0<a<1$, the backward 1-rarefaction wave $R_{U_l}^{1,-}$ issuing from $U_l$ intersects with the three dimensional surface $\Sigma_a$, i.e.,
\[
\Sigma_a:=\{U~|~\eta(U|U_l)=a\eta(U|U_r) \}.
\]
Since $dv\cdot r_1=1>0$, $v$ is strictly monotone along the integral curve of $r_1$, which means that the 1-rarefaction wave can be parameterized by $v$. Moreover since $d\lambda_1\cdot r_1>0$, $-r_1$ is the tangent vector of the backward 1-rarefaction wave $R_{U_l}^{1,-}$, which implies that $v$ decreases along $R_{U_l}^{1,-}$. That is, $v_+\le v_l$ for all parameters $v_+$ of $R_{U_l}^{1,-}$. Notice that $R_{U_l}^{1,-}$ is well-defined for all $v_+\in(0,v_l]$, because $-r_1(W)$ is smooth for all $W\in (0,\infty)\times (\bbr-\{0\})\times \bbr^2$. Indeed, since 
\beq\label{alpha-c}
\alpha_+ -c^2=\frac{1}{2} \Big[ \frac{B^2+\beta^2}{v} -c^2 +\sqrt{ \Big( \frac{B^2+\beta^2}{v}-c^2 \Big)^2+4\frac{B^2c^2}{v}} \Big]>0,
\eeq
we have
\begin{align}
\begin{aligned}\label{B-sign}
dB\cdot (-r_1)=\frac{(\alpha_{+}-c^2)}{B}=\left\{ \begin{array}{ll}
       >0\quad\mbox{if} ~B>0\\
        <0\quad\mbox{if} ~B<0, \end{array} \right.
\end{aligned}
\end{align}
which implies that $B_+\neq 0$ along $R_{U_l}^{1,-}$ due to $B_l\neq0$, thus $-r_1(W)$ is smooth for all $W\in (0,\infty)\times (\bbr-\{0\})\times \bbr^2$.\\
We now use the fact that
\begin{align*}
\begin{aligned}
&\mbox{for}~a<1,\quad\eta(U|U_l) \le a\eta(U|U_r) \quad\mbox{is equivalent to}\\
&\eta(U) \le \frac{1}{1-a} (\eta(U_l)-a\eta(U_r) -\nabla \eta(U_l)\cdot U_l + a\nabla \eta(U_r)\cdot U_r +(\nabla \eta(U_l)-a\nabla \eta(U_r))\cdot U),
\end{aligned} 
\end{align*}
which is rewritten as
\beq\label{quad-mhd}
\int_{v}^{\infty} p(s) ds + \frac{1}{2}(u^2 + w^2) + \frac{q^2}{2v}\le c_1+c_2 (v+q+u+w),  
\eeq
for some constants $c_1, c_2$. This implies that 
\[
\eta(U|U_l) \le a\eta(U|U_r) \Longleftrightarrow v>c_*~\mbox{and}~ |q|+ |u|+ |w|\le c^*\quad\mbox{for some constants}~c_*,~c^*>0, 
\]
since $\int_{0}^{\infty}p(s) ds =+\infty$, and the positive terms on $u$, $w$ and $q$ are quadratic in the left-hand side of \eqref{quad-mhd}. Therefore there exists $0<v_* \ll c_*$ such that
\[
\eta(R_{U_l}^{1,-}(v_*)|U_l) > a\eta(R_{U_l}^{1,-}(v_*)|U_r),
\]
which implies that $R_{U_l}^{1,-}$ intersects with $\Sigma_a$ for $a<1$, because $R_{U_l}^{1,-}$ is a continuous curve issuing from $U_l\in \{U~|~\eta(U|U_l) < a\eta(U|U_r)\}$.\\

On the other hand, we show that the forward 4-rarefaction wave $R_{U_r}^{4,+}$ issuing from $U_r$ intersects with the surface $\Sigma_a$ for any $a\ge 1$. Since $d\lambda_4\cdot r_4>0$ and $dv\cdot r_4<0$,
\beq\label{-r}
r_4~\mbox{ is the tangent vector of the forward 4-rarefaction wave}~ R_{U_r}^{4,+},
\eeq
and the parameter $v_+$ decreases along $R_{U_r}^{4,+}$. Moreover $R_{U_r}^{4,+}$ is well-defined for all $v_+\in(0,v_r]$ by the same reason as above.\\
Let us consider a continuous function 
\[
F_a(U):=\eta(U|U_l)-a\eta(U|U_r).
\]
We claim that 
\beq\label{positive-F}
F_1(R_{U_r}^{4,+}(v_*)) < 0 \quad\mbox{for some}~v_*\in(0,v_r].
\eeq
Using \eqref{-r}, we have
\begin{align*}
\begin{aligned}
\frac{dF_1(R_{U_r}^{4,+}(v_+))}{dv_+}&=(\nabla \eta(U_r)-\nabla \eta(U_l))\cdot \frac{d R_{U_r}^{4,+}(v_+)}{dv_+}\\
&=[p]+\Big[\frac{q^2}{2v^2}\Big] +\Big[\frac{q}{v}\Big] \frac{v(\alpha_+ -c^2)}{q_+} +[u]\sqrt{\alpha_+} -[w]\frac{v_+\beta(\alpha_{+}-c^2)}{q_+\sqrt{\alpha_+}}.
\end{aligned}
\end{align*}
And \eqref{MHD-RH} yields
\[
\frac{dF_1(R_{U_r}^{4,+}(v_+))}{dv_+}=\underbrace{(\sigma_2 +\sqrt{\alpha_+})[u]}_{I_1} + \underbrace{\frac{v_+(\alpha_+ -c^2)}{q_+}\Big(1- \frac{\beta^2}{\sigma_2\sqrt{\alpha_+} } \Big)\Big[\frac{q}{v}\Big]}_{I_2}. 
\] 
Since \eqref{alpha-c} yields
\beq\label{alpha-to}
\sqrt{\alpha_+} > \sqrt{-p^{\prime}(v_+)} =\sqrt{\gamma v_+^{-\gamma-1}} \to \infty\quad\mbox{as}~v_+\to0+,
\eeq
it follows from \eqref{2-shock} that $I_1 \to -\infty$ as $v_+\to0+$. \\
To control $I_2$, we use the condition \eqref{2-B}. Since
\begin{align*}
\begin{aligned}
dB\cdot r_4=\frac{(\alpha_{+}-c^2)}{B}=\left\{ \begin{array}{ll}
       >0\quad\mbox{if} ~B>0\\
        <0\quad\mbox{if} ~B<0, \end{array} \right.
\end{aligned}
\end{align*}
if $B_l> B_r>0$, $\Big[\frac{q}{v}\Big]=[B]<0$ and $q_+>0$. Since $\alpha_+\to+\infty$ by \eqref{alpha-to}, we have
\[
I_2 <0 \quad\mbox{for}~ v_+\ll 1.
\]
This is also true in the case of $B_l\le B_r<0$ because of $\Big[\frac{q}{v}\Big]>0$ and $q_+<0$.\\
Thus we have
\[
\frac{dF(R_{U_r}^{4,+}(v_+))}{dv_+} \to -\infty \quad\mbox{as}~v_+\to0+,
\]
which implies \eqref{positive-F}.\\
Therefore we conclude that $R_{U_r}^{4,+}$ intersects with $\Sigma_a$ for any $a\ge 1$, because $F_a(U_r)>0$ and $F_a(R_{U_r}^{4,+}(v_*))<F_1(R_{U_r}^{4,+}(v_*))<0$ for all $a\ge1$.\\

Hence for all $a>0$, the 2-shock wave $(U_l,U_r,\sigma_2)$ does not satisfies $a$-contraction property thank to Theorem \ref{thm-shock}.\\
Similarly we use the same arguments as above to show non-contraction for 3-shock wave $(\tilde{U}_l, \tilde{U}_r, \sigma_3)$ satisfying \eqref{3-B}. More precisely, we can show that the backward 1-rarefaction wave $R_{\tilde{U}_l}^{1,-}$ intersects with 
\[
\tilde{\Sigma}_a:=\{U~|~\eta(U|\tilde{U}_l)=a\eta(U|\tilde{U}_r) \}\quad \mbox{for any}~0<a\le 1,
\]
and the forward 4-rarefaction wave $R_{\tilde{U}_r}^{4,+}$ intersects with $\tilde{\Sigma}_a$ for any $a>1$. We omit the details.
\end{proof}

\subsection{Conditions via neighboring linearly degenerate fields}
The following criterion is on the case where there are linearly degenerate fields as neighboring families of the intermediate entropic shock. 
\begin{theorem}\label{thm-shock-2}
Consider the system \eqref{main} satisfying $(\mathcal{A}_1)$. Let $(u_l,u_r, \sigma_{l,r})$ be a given $i$-th the entropic shock such that $u_r=S_{u_l}^i(s_0)$, $\sigma_{u_l}(s_0)=\sigma_{l,r}$ for some $s_0>0$ and the conditions \eqref{Liu-1} and \eqref{Lax} are satisfied. Then the following statements holds.\\
\begin{itemize}
\item (1) For $0<a<1$, we assume that for some $j<i$, the $j$-th characteristic field is linearly degenerate, thus there exists the $j$-th Hugoniot curve $S_{u_l}^j(s)$ as the integral curve of the vector field $r_j$ with $S_{u_l}^j(0)=u_l$ such that the contact discontinuity $(u_l, S_{u_l}^j(s), \sigma_{u_l}^j(s))$ satisfies the Rankine-Hugoniot condition and
\beq\label{Liu-2}
\frac{d}{ds}\sigma_{u_l}^j (s) = 0,\quad \sigma_{u_l}^{j}(0)=\lambda_j(u_l).
\eeq
Moreover if there is $s_1>0$ such that
\beq\label{outer-0}
 \eta(S_{u_l}^j(s_1)|u_l) > a \eta(S_{u_l}^j(s_1)|u_r),
\eeq
then the entropic shock $(u_l,u_r, \sigma_{l,r})$ does not satisfy $a$-contraction.\\

\item (2) For $a>1$, we assume that for some $k>i$, the $k$-th characteristic field is linearly degenerate, thus there exists the $k$-th Hugoniot curve $S_{u_r}^k(s)$ as the integral curve of the vector field $r_k$ with $S_{u_r}^k(0)=u_r$ such that the contact discontinuity $(S_{u_r}^k(s), u_r, \sigma_{u_r}^k(s))$ satisfies the Rankine-Hugoniot condition and
\beq\label{Liu-3}
\frac{d}{ds}\sigma_{u_r}^k (s) = 0,\quad \sigma_{u_r}^{k}(0)=\lambda_k(u_r).
\eeq
Moreover if there is $s_2>0$ such that
\beq\label{inner-0}
 \eta(S_{u_r}^k(s_2)|u_l) < a \eta(S_{u_r}^k(s_2)|u_r),
\eeq
then the entropic shock $(u_l,u_r, \sigma_{l,r})$ does not satisfy $a$-contraction.\\

\item  (3) For $a=1$, we assume that one of the assumptions of (1) and (2) is satisfied. Then, the entropic shock $(u_l,u_r, \sigma_{l,r})$ does not satisfy $a$-contraction.\\
\end{itemize}
\end{theorem}

\begin{remark}
For example, the condition \eqref{outer-0} (resp. \eqref{inner-0}) is geometrically satisfied by a case that $C^1$ curve $S_{u_l}^j(s)$ (resp. $S_{u_r}^k(s)$) transversally intersects with the $(n-1)$-dimensional surface $\Sigma_a$.
\end{remark}
{\bf Proof of Theorem \ref{thm-shock-2}} 
{\bf(1) Case of $0<a<1$ :}
Let us put $u_-:=u_l$, $u_+ :=  S_{u_-}^j(s_1)$ and $\sigma_{\pm}:= \sigma_{u_-}^{j}(s_1)$, then the contact discontinuity $(u_-,u_+,\sigma_{\pm})$ satisfies
\[
\eta(u_-|u_l) =0 < a \eta(u_-|u_r)\quad \mbox{and}\quad\eta(u_+|u_l) > a \eta(u_+|u_r).
\]
We are going to show that the shock $(u_l,u_r, \sigma_{l,r})$ does not satisfies the second condition $(\mathcal{H} 2)$ of {\it{a-RES}} as
\begin{align*}
\begin{aligned}
D_{RH}(u_{l,r};u_{\pm}) := a q(u_+,u_r) - q(u_-,u_l) - \sigma_{\pm} ( a\eta(u_+|u_r) -\eta(u_-,u_l)) >0.
\end{aligned} 
\end{align*}
In fact, since $q(u_-,u_l)=\eta(u_-|u_l)=0$ by $u_-=u_l$, we have
\[
D_{RH}(u_{l,r};u_{\pm}) = a \Big(q(u_+,u_r) - \sigma_{\pm} \eta(u_+|u_r) \Big).
\]
We use the identity \eqref{metric} and the Rankine-Hugoniot condition for the discontinuity $(u_-,u_+,\sigma_{\pm})$, to get 
\begin{align*}
\begin{aligned}
D_{RH}(u_{l,r};u_{\pm}) &= a \Big[ q(u_+,u_l) + q(u_l,u_r) + (d\eta(u_l)-d\eta(u_r)) \cdot(f(u_+)-f(u_l)) \\
&\quad-\sigma_{\pm} \Big(\eta(u_+|u_l) + \eta(u_l|u_r) + (d\eta(u_l)-d\eta(u_r)) \cdot(u_+ -u_l) \Big) \Big]\\
&= a \Big[ q(u_+,u_l) + q(u_l,u_r)-\sigma_{\pm} \Big(\eta(u_+|u_l) + \eta(u_l|u_r) \Big) \Big],
\end{aligned}
\end{align*} 
where the fact $u_-=u_l$ is used above.\\
We decompose the above relation into three parts by
\begin{align*}
\begin{aligned}
D_{RH}(u_{l,r};u_{\pm}) &= a \Big( q(u_l,u_r) -\lambda_{j}(u_l) \eta(u_l|u_r) \Big) +a(\lambda_{j}(u_l) -\sigma_{\pm})\eta(u_l|u_r) \\
&\quad + a \Big( q(u_+,u_l) -\sigma_{\pm} \eta(u_+|u_l) \Big)\\
&=: I_1 +I_2 +I_3.
\end{aligned}
\end{align*} 
Following the same argument as \eqref{since-1} by using \eqref{vasseur} and the fact $\lambda_j(u_l)<\lambda_i(u_l)$, we have
\begin{align*}
\begin{aligned}
I_1=a \Big[\Big( q(u_l,u_r) -\sigma_{l,r}\eta(u_l|u_r) \Big) +(\sigma_{l,r} - \lambda_{j}(u_l)) \eta(u_l|u_r) \Big] >0.
\end{aligned}
\end{align*} 
Since $\lambda_{j}(u_l) =\sigma_{u_l}^j(0)=\sigma_{u_l}^j(s_1)=\sigma_{\pm}$ by \eqref{Liu-2}, we have
\[
I_2 = a(\lambda_{j}(u_l) -\sigma_{\pm})\eta(u_l|u_r)=0.
\]
We use \eqref{vasseur} together with \eqref{Liu-2} to get
\[
I_3 = a \Big( q(u_+,u_-) -\sigma_{\pm} \eta(u_+|u_-)\Big) =a \int_0^{s_1} \frac{d}{ds}\sigma_{u_-}^j(s)\eta(u_-|S_{u_-}^j(s))ds =0.
\]
Therefore we have 
\[D_{RH}(u_{l,r};u_{\pm}) >0.\]
Hence the shock $(u_l,u_r, \sigma_{l,r})$ is not {\it{a-RES}}, which provides the conclusion by Theorem \ref{thm-cont-2}.\\
{\bf(2) Case of $a>1$ :} We follow the same argument as above by considering $S_{u_r}^{k}$ instead of $S_{u_l}^{j}$. By \eqref{inner-0}, we have 
\[
\eta(u_-|u_l)  < a \eta(u_-|u_r)\quad \mbox{and}\quad\eta(u_+|u_l) > 0=a \eta(u_+|u_r).
\]
for the contact discontinuity $(u_-,u_+,\sigma_{\pm})$ where $u_+:=u_r$, $u_- :=  S_{u_+}^k(s_2)$ and $\sigma_{\pm}:= \sigma_{u_+}^{k}(s_2)$.\\
Since $q(u_+,u_r)=\eta(u_+,u_r)=0$ by $u_+=u_r$, we have
\[
D_{RH}(u_{l,r};u_{\pm}) = - q(u_-,u_l) + \sigma_{\pm} \eta(u_-,u_l).
\]
Using \eqref{metric} and the Rankine-Hugoniot condition of the discontinuity $(u_-,u_+,\sigma_{\pm})$, we have
\begin{align*}
\begin{aligned}
D_{RH}(u_{l,r};u_{\pm}) = - q(u_-,u_+) - q(u_r,u_l)+\sigma_{\pm} \Big(\eta(u_-|u_+) + \eta(u_r|u_l) \Big) ,
\end{aligned}
\end{align*} 
where the fact $u_+=u_r$ is used above.\\
We decompose the above relation into three parts by
\begin{align*}
\begin{aligned}
D_{RH}(u_{l,r};u_{\pm}) &= - \Big( q(u_r,u_l) -\lambda_{k}(u_r) \eta(u_r|u_l) \Big) +(\sigma_{\pm}-\lambda_{k}(u_r) )
\eta(u_r|u_l) \\
&\quad - \Big( q(u_-,u_+) -\sigma_{\pm} \eta(u_-|u_+) \Big)\\
&=: J_1 +J_2 +J_3.
\end{aligned}
\end{align*} 
By the same argument as the first case, we have
\begin{align*}
\begin{aligned}
J_1=- \Big( q(u_r,u_l) -\sigma_{l,r}\eta(u_r,u_l) \Big) +( \lambda_{k}(u_r)-\sigma_{l,r} ) \eta(u_r,u_l)  >0.
\end{aligned}
\end{align*} 
Since $\lambda_{k}(u_r) =\sigma_{u_r}^k(0)=\sigma_{u_r}^k(s_2)=\sigma_{\pm}$ by \eqref{Liu-3}, we have $J_2=0$.\\
We use \eqref{vasseur} together with \eqref{Liu-3} to get
\[
J_3 =- \int_0^{s_2} \frac{d}{ds}\sigma_{u_+}^k(s)\eta(u_+|S_{u_+}^k(s))ds =0.
\]
Therefore we have 
\[D_{RH}(u_{l,r};u_{\pm}) >0.\]
{\bf(3) Case of $a=1$ :} Since the proof of two cases above does not depend on the strength of $a$, if one of the assumptions of (1) and (2) is satisfied, we end up with no $a$-contraction of the entropic shock $(u_l,u_r, \sigma_{l,r})$.

\section{Criteria preventing  $a$-contractions for intermediate contact discontinuities}\label{no-cont-inter}
In this section, as an application of Theorem \ref{thm-cont-2}, we consider some sufficient conditions, which provides no $a$-contraction for intermediate contact discontinuities. Then we apply the sufficient condition to the full Euler system in order to find a range of weights $a$, on which that intermediate contact discontinuities do not satisfy $a$-contraction. 

\subsection{Condition preventing $a$-contractions for intermediate contact discontinuities}
In this part, we consider the case that the system \eqref{main} has a genuinely nonlinear field, which has same conditions as Theorem \ref{thm-shock}. In fact, the proof of Theorem \ref{thm-shock} still holds in the case of the contact discontinuity $(u_l,u_r, \sigma_{l,r})$ instead of shock. Thus we state the following theorem without proof.
\bigskip

\begin{itemize}
\item $(\mathcal{A}_2)$ : Suppose that for some $1< i < n$, the $i$-th characteristic field $(\lambda_i, r_i)$ is linearly degenerate. Then, for a given $i$-th contact discontinuity $(u_l,u_r,\sigma_{l,r})$, there exists the $i$-th Hugoniot curve $S_{u_l}^i(s)$ as the integral curve
of the vector field $r_i$ such that
$S_{u_l}^i(0)=u_l$ and $S_{u_l}^i(s_0)=u_r$ for some $s_0>0$, and the Rankine-Hugoniot condition:
\[f(S_{u_l}^i(s)) -f(u_l) = \sigma_{u_l}^i(s) (S_{u_l}^i(s) - u_l).\]
where $\sigma_{u_l}^i(s)$ is a velocity function satisfying the Liu entropy condition:
\beq\label{Liu}
\frac{d}{ds}\sigma_{u_l}^{i}(s) = 0, \quad 0\le s\le s_0,
\eeq
and $\sigma_{u_l}^i(0)=\lambda_i(u_l)$, $\sigma_{u_l}^i(s_0)=\sigma_{l,r}$.\\
We also suppose the extended Lax condition:
\beq\label{Lax-00}
\lambda_{l-1} (u_-) \le \sigma \le \lambda_{l+1} (u_+),\quad \mbox{for any $l$-th shock $(u_-,u_+,\sigma)$}.
\eeq
\end{itemize}

\begin{theorem}\label{thm-contact-2}
Consider the system \eqref{main} satisfying $(\mathcal{A}_2)$ and \eqref{Lax-00}. Let $(u_l,u_r, \sigma_{l,r})$ be a given $i$-th intermediate contact discontinuity (i.e., $1<i<n$) such that $u_r=S_{u_l}^i(s_0)$, $\sigma_{u_l}^i(s_0)=\sigma_{l,r}$ for some $s_0>0$. Then the following statements holds.\\
\begin{itemize}
\item (1) For $0<a<1$, we assume that there is a $C^1$ $j$-th rarefaction curve $R_{u_l}^j (s)$ with $j<i$ such that $R_{u_l}^j(0)=u_l$ and the backward curve $R_{u_l}^{j,-} (s)$ of $R_{u_l}^j (s)$, i.e., $\lambda_j (R_{u_l}^{j,-} (s)) < \lambda_j (u_l) $, intersects with the $(n-1)$-dimensional surface $\Sigma_a$.
Then, the contact discontinuity $(u_l,u_r, \sigma_{l,r})$ does not satisfy $a$-contraction.\\

\item (2) For $a>1$, we assume that there is a $C^1$ $k$-th rarefaction curve $R_{u_r}^k (s)$ with $k>i$ such that $R_{u_r}^k(0)=u_r$ and the forward curve $R_{u_r}^{k,+} (s)$ of $R_{u_r}^k (s)$, i.e., $\lambda_k (R_{u_r}^{k,+} (s)) > \lambda_k (u_r) $, intersects with the $(n-1)$-dimensional surface $\Sigma_a$.
Then, the contact discontinuity $(u_l,u_r, \sigma_{l,r})$ does not satisfy $a$-contraction.\\

\item  (3) For $a=1$, we assume that one of the assumptions of (1) and (2) is satisfied. Then, the contact discontinuity $(u_l,u_r, \sigma_{l,r})$ does not satisfy $a$-contraction.\\
\end{itemize}
\end{theorem}

{Application to gas dynamics}
As an application of Theorem \ref{thm-contact-2}, we here find weights $a>0$, on which the contact discontinuity of full Euler system for a perfect gas does not satisfy $a$-contraction. The $3\times 3$ full Euler system reads
\begin{align}
\begin{aligned}\label{euler}
\left\{ \begin{array}{ll}
       \partial_t \rho + \partial_x(\rho v) =0\\
        \partial_t (\rho v) + \partial_x (\rho v^2 + p) = 0\\
         \partial_t (\rho (\frac{1}{2}v^2 + e)) + \partial_x (\rho (\frac{1}{2}v^2 + e)v + pv) = 0, \end{array} \right.
\end{aligned}
\end{align}
Here, the equation of state for a perfect gas is given by
\beq\label{pressure}
p=(\gamma -1)\rho e,
\eeq
where $\gamma > 1$.\\ 
We consider the conservative variables $\rho$, $q:=\rho v$ and $\mathcal{E}:=\rho (\frac{1}{2}v^2 + e)$ and put $u:=(\rho, q, \mathcal{E})$, and the entropy
\beq\label{euler-ent}
\eta (u) = (\gamma -1) \rho \mbox{ln} \rho -\rho \mbox{ln} e,
\eeq
where $e$ is given by $e=\frac{\mathcal{E}}{\rho} - \frac{q^2}{2\rho^2}$ in conservative variables.\\
By a straightforward computation (See for example \cite{Serre_book}), the three characteristic fields are given as
\begin{align*}
\begin{aligned}
&\lambda_1=v-c,\quad  r_1=\Big(-\rho, c, -\rho^{-1}p \Big)^T,\\
&\lambda_2=v,\quad  r_2=\Big(-\partial_e p, 0, -\partial_{\rho}p \Big)^T,\\
&\lambda_3=v+c,\quad  r_3=\Big(\rho, c, \rho^{-1}p \Big)^T,
\end{aligned}
\end{align*}
where $c:=\sqrt{\gamma(\gamma-1)e}$ denotes the sound speed.\\
Thus since $d\lambda_i\cdot r_i>0$ for $i=1,3$ and $d\lambda_2\cdot r_2=0$, the first and third characteristic fields are genuinely non-linear, whereas the second characteristic field is linearly degenerate. Let $(u_l, u_r)$ be 2-contact discontinuity, then we have
\beq\label{second}
v_l = v_r,\quad p_l =p_r.
\eeq
We refer to \cite{Serre_book} for these relations.\\

\begin{theorem}\label{thm-euler}
Let $(u_l, u_r)$ be the 2-contact discontinuity for the system \eqref{euler}-\eqref{pressure}. Then, there is no weight $a$ for the range of either
\begin{align}
\begin{aligned}\label{a-range}
&0<a < \frac{e_r}{e_l}, ~a>1, \quad \mbox{when} ~e_l>e_r,\quad \mbox{or}\\
& a<1, ~a > \frac{e_r}{e_l},\quad \mbox{when} ~ e_l<e_r, 
\end{aligned} 
\end{align}
such that $(u_l,u_r)$ satisfies $a$-contraction.
\end{theorem}
\begin{proof}
We show that the assumptions of Theorem \ref{thm-contact-2} are satisfied for $a$ belonging to \eqref{a-range}. For that, we separate the range of weight $a>0$ into two cases.\\
{\bf i) Case of either $0<a < \frac{e_r}{e_l}$ when $e_l>e_r$ or $a<1$ when $e_l<e_r$}.
For such $a$, we are going to show that the first backward rarefaction wave $R_{u_l}^{1,-}$ issuing from $u_l$ intersects with the surface $\Sigma_a$.  
As in \cite{Serre_book}, it turned out that the first backward rarefaction wave $R_{u_l}^{1,-}$ starting from $u_l$ can be parametrized by the pressure $p_+$ as follows.
\begin{align}
\begin{aligned}\label{1-rare}
&R_{u_l}^{1,-}~\mbox{is parametrized as}~p_l \le p_+,\\
&\rho_+=\phi(p_+ ; \rho_l,p_l), \quad \mbox{where $\phi$ is increasing in the first argument},\\
& v_+ = v_l - \int_{p_l}^{p_+} \frac{dp}{\rho c},
\end{aligned} 
\end{align}
where $c$ is the speed of sound, in particular $c=\sqrt{\gamma(\gamma-1)e}$ for a perfect gas. That is, along $R_{u_l}^{1,-}$, the pressure $p_+$ and density $\rho_+$ increase while the velocity $v_+$ decrease. Then, using 
\beq\label{rho-c}
\rho c = \rho \sqrt{\gamma(\gamma-1)e} = \sqrt{\gamma\rho p} \quad\mbox{by \eqref{pressure} and \eqref{1-rare}}, 
\eeq
we have
\[
v_+ =v_l  -\frac{1}{\sqrt{\gamma}}\int_{p_l}^{p_+} \frac{dp}{\sqrt{\rho p}}.
\]
Multiplying by $\sqrt{\rho_+}$, we have
\begin{align*}
\begin{aligned}
\sqrt{\rho_+}v_+  &= \sqrt{\rho_+}v_l -\frac{1}{\sqrt{\gamma}}\int_{p_l}^{p_+} \sqrt{\frac{\rho_+}{\rho}}\frac{dp}{\sqrt{ p}}\\
&\le \sqrt{\rho_+}v_l -\frac{1}{\sqrt{\gamma}}\int_{p_l}^{p_+} \frac{dp}{\sqrt{ p}}\\
&=\sqrt{\rho_+}v_l -\frac{2}{\sqrt{\gamma}} (\sqrt{p_+}-\sqrt{p_l}),
\end{aligned} 
\end{align*}
where we have used $\rho_+ \ge \rho$ thanks to \eqref{1-rare}.\\
If $\rho_+$ is bounded in $p_+$, then we have
\[
\sqrt{\rho_+}v_+ \rightarrow -\infty\quad \mbox{as}~p_+\rightarrow \infty,
\]
which implies
\beq\label{energy-2}
\mbox{either}\quad\rho_+\rightarrow \infty\quad\mbox{or} \quad\mathcal{E}_+\rightarrow \infty\quad \mbox{as}~p_+\rightarrow \infty.
\eeq
This means that the either density or energy should be unbounded along the first backward rarefaction wave $R_{u_l}^{1,-}$ issuing from $u_l$.\\
On the other hand, it turns out that $u_l$ is enclosed by the surface $\Sigma_a$ in the case of either $0<a < \frac{e_r}{e_l}$ when $e_l>e_r$ or $a<1$ when $e_l<e_r$ as follows.\\
Recall \eqref{a-1}, i.e.,
\begin{align*}
\begin{aligned}
&\mbox{for}~a<1,\quad\eta(u|u_l) \le a\eta(u|u_r) \quad\mbox{is equivalent to}\\
&\eta(u) \le \frac{1}{1-a} (\eta(u_l)-a\eta(u_r) -\nabla \eta(u_l)\cdot u_l + a\nabla \eta(u_r)\cdot u_r +(\nabla \eta(u_l)-a\nabla \eta(u_r))\cdot u).
\end{aligned} 
\end{align*}
Since the entropy \eqref{euler-ent} can be rewritten as
\beq\label{ent-1}
\eta (u) = \gamma\rho\mbox{ln}\rho -\rho \mbox{ln}( \mathcal{E} -\frac{q^2}{2\rho}),
\eeq
using $\rho e = \mathcal{E} -\frac{q^2}{2\rho}$, we have
\begin{align}
\begin{aligned}\label{diff}
&\partial_{\rho}\eta (u) = \gamma (\mbox{ln}\rho  +1) -\mbox{ln}(\rho e) - \frac{v^2}{2e},  \\
&\partial_{q}\eta (u) = \frac{v}{e},\\ 
&\partial_{\mathcal{E}}\eta (u) = -\frac{1}{e}.
\end{aligned} 
\end{align}
We use \eqref{pressure}, \eqref{second} and \eqref{diff} to compute
\begin{align}
\begin{aligned}\label{nabla-eta}
(\nabla \eta(u_l)-a\nabla \eta(u_r))\cdot u &= (\partial_{\rho}\eta(u_l)-a\partial_{\rho}\eta(u_r)) \rho\\
 &\quad + (\partial_{q}\eta(u_l)-a\partial_{q}\eta(u_r)) q + \Big(\frac{a}{e_r}-\frac{1}{e_l}\Big)\mathcal{E}.
\end{aligned} 
\end{align}
By \eqref{a-1}, \eqref{ent-1} and \eqref{nabla-eta}, we have
\begin{align*}
\begin{aligned}
\Big(\gamma\mbox{ln}\rho - \mbox{ln}( \mathcal{E} -\frac{q^2}{2\rho}) - c_1 \Big)\rho+c_3\mathcal{E} \le c_2 q + c_4,
\end{aligned} 
\end{align*}
where 
\begin{align*}
\begin{aligned}
& c_1 = \frac{1}{1-a}(\partial_{\rho}\eta(u_l)-a\partial_{\rho}\eta(u_r)),\\
& c_2=\frac{1}{1-a}(\partial_{q}\eta(u_l)-a\partial_{q}\eta(u_r)),\\
&c_3=-\frac{1}{1-a} \Big(\frac{a}{e_r}-\frac{1}{e_l}\Big),\\
&c_4 =   \frac{1}{1-a} (\eta(u_l)-a\eta(u_r) -\nabla \eta(u_l)\cdot u_l + a\nabla \eta(u_r)\cdot u_r).
\end{aligned} 
\end{align*}
Since $c_3>0$ in the case of either $0<a < \frac{e_r}{e_l}$ when $e_l>e_r$ or $a<1$ when $e_l<e_r$, we use $c_2 q\le \frac{c_3 \rho v^2}{4}+ c_5\rho $ by Young's inequality, to get
\[
\underbrace{\Big(\gamma\mbox{ln}\rho - \mbox{ln}( \mathcal{E} -\frac{q^2}{2\rho}) - c_1-c_5 \Big)\rho+\frac{c_3}{2}\mathcal{E}}_{L} \le  c_4.  
\]
Notice that $c_4$ is constant while
\[
L\to \infty \quad \mbox{as} ~\mathcal{E} \to \infty~\mbox{or} ~\rho\to \infty,
\]
which implies that $\{u~|~\eta(u|u_l) \le a\eta(u|u_r)\}$ is bounded. Thus $u_l$ is enclosed by $\Sigma_a$.
Therefore by this and \eqref{energy-2}, we conclude that the first backward rarefaction wave $R_{u_l}^{1,-}$ has to intersect with the $(n-1)$-dimensional surface $\Sigma_a$.\\

 {\bf ii) Case of either $a>1$ when $e_l>e_r$ or $a > \frac{e_r}{e_l}$ when $e_l<e_r$}.
For such $a$, we use the same argument as previous to show that the third forward rarefaction wave $R_{u_r}^{3,+}$ issuing from $u_r$ intersects with the surface $\Sigma_a$. Since the third forward rarefaction wave $R_{u_r}^{3,+}$ starting from $u_r$ can be parametrized by the pressure $p_+$ as:
\begin{align}
\begin{aligned}\label{3-rare}
&R_{u_r}^{3,+}~\mbox{is parametrized as}~p_r \le p_+,\\
&\rho_+=\phi(p_+ ; \rho_r,p_r), \quad \mbox{where $\phi$ is increasing in the first argument},\\
& v_+ = v_r + \int_{p_r}^{p_+} \frac{dp}{\rho c},
\end{aligned} 
\end{align}
the pressure $p_+$, density $\rho_+$ and the velocity $v_+$ all increase along $R_{u_r}^{3,+}$.
Then by \eqref{rho-c}, 
\[
v_+ - v_r = \frac{1}{\sqrt{\gamma}}\int_{p_r}^{p_+} \frac{dp}{\sqrt{\rho p}}.
\]
Multiplying by $\sqrt{\rho_+}$, we have
\begin{align*}
\begin{aligned}
\sqrt{\rho_+}v_+ &=\sqrt{\rho_+}v_r + \frac{1}{\sqrt{\gamma}}\int_{p_r}^{p_+} \sqrt{\frac{\rho_+}{\rho}}\frac{dp}{\sqrt{ p}}\\
&\ge \sqrt{\rho_+}v_r +  \frac{1}{\sqrt{\gamma}}\int_{p_r}^{p_+} \frac{dp}{\sqrt{ p}}\\
&=\sqrt{\rho_+}v_r +\frac{2}{\sqrt{\gamma}} (\sqrt{p_+}-\sqrt{p_r}),
\end{aligned} 
\end{align*}
where we have used $\rho_+ \ge \rho$ thanks to \eqref{3-rare}. If $\rho_+$ is bounded in $p_+$, then we have
\[
\sqrt{\rho_+}v_+ \rightarrow \infty\quad \mbox{as}~p_+\rightarrow \infty,
\]
which implies
\beq\label{energy-3}
\mbox{either}\quad\rho_+\rightarrow \infty\quad\mbox{or} \quad\mathcal{E}_+\rightarrow \infty\quad \mbox{as}~p_+\rightarrow \infty.
\eeq
This means that the either density or energy should be unbounded along the third forward rarefaction wave $R_{u_r}^{3,+}$ issuing from $u_r$.\\
On the other hand, by the same computation as previous case, $u_r$ is enclosed by $\Sigma_a$. Indeed, for $a>1$, $\eta(u|u_l) \ge a\eta(u|u_r)$ is equivalent to \eqref{a-1}, besides $c_3>0$ in the case of either $a>1$ when $e_l>e_r$ or $a > \frac{e_r}{e_l}$ when $e_l<e_r$. Therefore by this and \eqref{energy-3}, we conclude that the first backward rarefaction wave $R_{u_r}^{3,+}$ has to intersect with $\Sigma_a$.
\end{proof}

\begin{remark}\label{remark-serre}
Recently in \cite{Serre-Vasseur_2015}, the authors has showed that the 2-contact discontinuity of full Euler system \eqref{euler} satisfies $a$-contraction for the specific weight $a=\frac{e_r}{e_l}$ in the borderline of weights as in \eqref{a-range}. More precisely, Euler system is written in Lagrangian mass coordinates is considered and it turns out that the 2-contact discontinuity $(u_l,u_r)$ satisfies $\frac{\theta_r}{\theta_l}$-contraction, where $\theta_l$ and $\theta_r$ denote the temperature corresponding to the left end state $u_l$ and right one $u_r$, respectively. In fact, since $e=c\theta$ for some positive constant $c$ for the perfect gas, the weight $a=\frac{\theta_r}{\theta_l}$ is equal to $a=\frac{e_r}{e_l}$, moreover, since it is known that the $a$-contraction property is invariant under a transformation between the Eulerian coordinates and Lagrangian coordinates in \cite{Serre-Vasseur} (See also \cite{Serre-1991} and \cite{Wagner}), Serre's result also holds full Euler system \eqref{euler} in Eulerian coordinates. 
\end{remark}

\section{Appendix}
We here give the proof of \eqref{delta-k} in Lemma \ref{lem-v1}. Since $\frac{d}{ds}\sigma_{u}^1(s)$ and $\frac{d}{ds}\eta(u|S_{u}^1(s))$ are both continuous and non zero at $s=s_0$ by \eqref{Liu-0}, there exists $\delta\in(0,s_0)$ such that for all $s$ with $|s-s_0|<\delta$,
\begin{align}
\begin{aligned}\label{help}
&\Big|\frac{d}{ds}\sigma_{u}^1(s)-\frac{d}{ds}\sigma_{u}^1(s_0)\Big| \le \frac{1}{2}\Big|\frac{d}{ds}\sigma_{u}^1(s_0)\Big|,\\
&\Big|\frac{d}{ds}\eta(u|S_{u}^1(s))-\frac{d}{ds}\eta(u|S_{u}^1(s_0)) \Big| \le \frac{1}{2} \Big| \frac{d}{ds}\eta(u|S_{u}^1(s_0)) \Big|,
\end{aligned} 
\end{align} 
which can be estimates as
\begin{align}
\begin{aligned}\label{help-0}
&\frac{d}{dt}\sigma^1_u(s) = - \Big|\frac{d}{dt}\sigma^1_u(s)\Big| \le \frac{1}{2} \Big|\frac{d}{dt}\sigma^1_u(s_0)\Big|,\\
&\frac{d}{ds}\eta(u|S_{u}^1(s))=\Big|\frac{d}{ds}\eta(u|S_{u}^1(s))\Big| \ge \frac{1}{2} \Big|\frac{d}{ds}\eta(u|S_{u}^1(s_0))\Big|.
\end{aligned} 
\end{align} 
Thus, using \eqref{vasseur-1} and \eqref{help-0}, we have that for all $s$ with $|s-s_0|<\delta$,
\begin{align*}
\begin{aligned}
q(S_u^1(s),S_u^1(s_0))   -  \sigma_u^1(s)  \eta(S_u^1(s)|S_u^1(s_0)) &= \int^s_{s_0} \frac{d}{dt}\sigma^1_u(t)\Big( \eta(u|S_u^1(t))-\eta(u|S_u^1(s_0)) \Big) dt\\
&\le -\frac{1}{4}\Big|\frac{d}{dt}\sigma^1_u(s_0)\Big| \frac{d}{ds}\eta(u|S_{u}^1(s_0)) \int^s_{s_0}  (t-s_0) dt\\
&\le -\frac{1}{8}\Big|\frac{d}{dt}\sigma^1_u(s_0)\Big| \frac{d}{ds}\eta(u|S_{u}^1(s_0)) |s-s_0|^2.
\end{aligned} 
\end{align*} 
Since it follows from \eqref{help} that 
\[
\Big|\frac{d}{ds}\sigma_{u}^1(s)\Big| \le \frac{3}{2}\Big|\frac{d}{ds}\sigma_{u}^1(s_0)\Big|,
\]
we have
\[
|\sigma^1_u(s)-\sigma^1_u(s_0)| \le \frac{3}{2}\Big|\frac{d}{ds}\sigma_{u}^1(s_0)\Big| |s-s_0|,
\]
which yields that for all $s$ with $|s-s_0|<\delta$,
\begin{align*}
\begin{aligned}
q(S_u^1(s),S_u^1(s_0))   -  \sigma_u^1(s)  \eta(S_u^1(s)|S_u^1(s_0)) \le -k_1 |\sigma^1_u(s)-\sigma^1_u(s_0)|^2,
\end{aligned} 
\end{align*} 
where $k_1$ is a positive constant as
\[
k_1 =\min_{u\in B} \Big( \frac{1}{18}\Big|\frac{d}{ds}\sigma_{u}^1(s_0)\Big|^{-1}  \frac{d}{ds}\eta(u|S_{u}^1(s_0)) \Big).
\]
On the other hand, let us show \eqref{delta-k} for $|s-s_0|\ge\delta$. For all $s$ with $s\le s_0-\delta$, we use \eqref{vasseur-1} to get
\begin{align*}
\begin{aligned}
q(S_u^1(s),S_u^1(s_0))   -  \sigma_u^1(s)  \eta(S_u^1(s)|S_u^1(s_0)) &= \int_{s}^{s_0-\delta} \frac{d}{dt}\sigma^1_u(t)\Big( \eta(u|S_u^1(s_0))-\eta(u|S_u^1(t)) \Big) dt\\
&\quad +\int_{s_0-\delta}^{s_0} \frac{d}{dt}\sigma^1_u(t)\Big( \eta(u|S_u^1(s_0))-\eta(u|S_u^1(t)) \Big) dt\\
&=:I_1 + I_2.
\end{aligned} 
\end{align*} 
If we consider a positive constant $c_1$ satisfying
\beq\label{small-c1}
c_1 \le \min_{u\in B}  \Big(  \eta(u|S_u^1(s_0))-\eta(u|S_u^1(s_0-\delta)) \Big),
\eeq
using \eqref{Liu-0}, we have
\begin{align*}
\begin{aligned}
I_1 &\le \int_{s}^{s_0-\delta}  \frac{d}{dt}\sigma^1_u(t) \Big(  \eta(u|S_u^1(s_0))-\eta(u|S_u^1(s_0-\delta)) \Big) dt \\
&\le -c_1| \sigma(s_0-\delta)-\sigma(s) |\\
&\le -c_1| \sigma(s_0)-\sigma(s) | + c_1|\sigma(s_0)- \sigma(s_0-\delta) |.
\end{aligned} 
\end{align*} 
Since $I_2<0$ for all $u\in\mathcal{U}$, we choose $c_1$ sufficiently small such that \eqref{small-c1} and 
\[
c_1|\sigma(s_0)- \sigma(s_0-\delta) | \le -\min_{u\in B}I_2,
\]
which yields
\begin{align*}
\begin{aligned}
q(S_u^1(s),S_u^1(s_0))   -  \sigma_u^1(s)  \eta(S_u^1(s)|S_u^1(s_0)) \le -c_1| \sigma(s_0)-\sigma(s) |.
\end{aligned} 
\end{align*} 
Similarly, choosing $c_2>0$ sufficiently small such that
\begin{align*}
\begin{aligned}
&c_2 \le \min_{u\in B}  \Big(  \eta(u|S_u^1(s_0+\delta))-\eta(u|S_u^1(s_0)) \Big),\\
&c_2|\sigma(s_0)- \sigma(s_0+\delta) | \le -\min_{u\in B}\int_{s_0}^{s_0+\delta} \frac{d}{dt}\sigma^1_u(t)\Big(\eta(u|S_u^1(t))- \eta(u|S_u^1(s_0)) \Big) dt,
\end{aligned} 
\end{align*} 
we have that for all $s$ with $s\ge s_0+\delta$,
\begin{align*}
\begin{aligned}
q(S_u^1(s),S_u^1(s_0))   -  \sigma_u^1(s)  \eta(S_u^1(s)|S_u^1(s_0)) \le -c_2| \sigma(s_0)-\sigma(s) |.
\end{aligned} 
\end{align*}


\bibliography{Kang-Vasseur2015}
\end{document}